\documentclass[11pt]{article}
\usepackage[margin=1.15in]{geometry}
\usepackage{amsmath,amssymb,amsthm,mathtools}
\usepackage{mathrsfs}
\usepackage{microtype}
\usepackage{enumitem}
\setlist{itemsep=2pt,topsep=4pt}
\usepackage[colorlinks=true,linkcolor=blue,citecolor=blue,urlcolor=blue]{hyperref}
\usepackage{float}

\newcommand{\lam}{\lambda}

\newcommand{\sh}{\sinh}

\newtheorem{theorem}{Theorem}[section]
\newtheorem{lemma}[theorem]{Lemma}
\newtheorem{proposition}[theorem]{Proposition}
\newtheorem{definition}[theorem]{Definition}
\newtheorem{remark}[theorem]{Remark}
\newtheorem{corollary}[theorem]{Corollary}

\newtheorem{example}[theorem]{Example}

\DeclareMathOperator{\Area}{Area}

\newcommand{\cB}{\mathcal B}
\newcommand{\cK}{\mathcal K}

\newcommand{\pair}[2]{\left\langle #1,#2\right\rangle}
\newcommand{\R}{\mathbb R}
\newcommand{\LLp}{\mathcal L^+}
\newcommand{\osp}{{\rm OSp}(1|2)}

\def\hyperp{{\rm I}\kern-.3ex{\rm H}}

\title{Regular  Additive Scalar Invariants of ${\mathcal N}=1$\\ Hyperbolic Spin Supertriangles}

\author{ Ilia Gaiur\\ Institut des Hautes \'Etudes Scientifiques\\ 35 route de Chartres, 91440 Bures-sur-Yvette, France\\ \texttt{gaiur@ihes.fr} \and Robert Penner\\ Institut des Hautes \'Etudes Scientifiques\\ 35 route de Chartres, 91440 Bures-sur-Yvette, France\\ \texttt{rpenner@ihes.fr} }

\date{\today}

\begin{document}
\maketitle

\begin{abstract}
\noindent We classify regular cyclic scalar invariants of non-degenerate hyperbolic
${\mathcal N}=1$ spin supertriangles which are additive under geodesic
subdivision.  After fixing the spin lift, the common-first gauge (introduced in
earlier work by the second-named author and further developed here) gives a
complete $3|4$ modulus of such supertriangles..  The space of additive scalar invariants is proved to have
superdimension $3|1$: ordinary hyperbolic area in degree zero, two even
generators in degree two, and one odd generator in degree three; the
degree-one and degree-four sectors vanish.  One degree-two generator is the
fermionic first variation of body area, equivalently the nilpotent correction
from ordinary body angle defect to super angle defect.  The degree-three
generator changes sign under reversal of the spin lift, while the other
generators are fixed, so ordinary supertriangles carry only
three even invariants.
\end{abstract}

%\tableofcontents

\section*{Introduction}

For ordinary hyperbolic triangles, the continuous isometry-invariant scalar
valuations which are additive under geodesic subdivision are up to overall scale
generated by hyperbolic area
\cite{Klain2006HyperbolicValuations}.  Angle defect 
$
\pi-(A+B+C)
$,
given by subtracting from pi the sum of interior angles $A,B,C$ at the respective vertices $P,Q,R$,
is also easily seen to be additive, and the classical Angle Defect Theorem identifies it with
area.  This uniqueness underlies Lambert's original proof \cite{Lambert} of this identity.

The situation is different for hyperbolic supertriangles.
The calculation in \cite{defect}, when expressed in a suitable gauge,
produced a non-vanishing degree-two correction between super angle defect and
area.  As explained in \cite{corrigendum}, the area form used in \cite{defect} is not
invariant under the odd part of $\osp$, so that comparison must be understood
gaugewise rather than as an identity of two global invariants.  The intrinsic
degree-two quantity is instead the difference between super angle defect and
body angle defect, or equivalently the fermionic first variation of body area discussed here.
In any case, this raises the question of finding all additive invariants of hyperbolic
supertriangles, which is solved in the present paper.
 
The common-first gauge developed below makes the first fermionic coordinate
of the three vertices equal. In this gauge, the four odd
parameters are
$
(\lambda,\beta_P,\beta_Q,\beta_R),$
and the central fermionic involution acts by
\[
\iota:
(\lambda,\beta_P,\beta_Q,\beta_R)
\longmapsto
(-\lambda,-\beta_P,-\beta_Q,-\beta_R).
\]
We principally work with
\emph{spin supertriangles}, meaning supertriangles together with a fixed
choice of signs on the fermions.  Ordinary supertriangles are obtained by quotienting by
$\iota$.
There is the following essential point about spin structure: 
no nonzero odd scalar invariant can be defined on the quotient which forgets this sign
since a homogeneous expression of fermionic degree $k$ is multiplied by
$(-1)^k$.

 After the spin lift is fixed, the resulting
four signed fermions together with the three even invariants given by pairwise 
distances between vertices $P,Q,R$ form a
complete $3|4$ modulus of hyperbolic spin supertriangles.  Every scalar function therefore has a finite
expansion 
\[
F=F_0+F_1+F_2+F_3+F_4
\]
by total fermionic degree, where $F_k$ is a linear combination of homogeneous expressions
with coefficients given by functions of the the bodies of the three pairwise vertex distances,
and the classification can be carried out fermionic degree by degree in these coordinates. 
 
Our main result is that any such additive scalar invariant is given by
\begin{equation}\label{eq:mainclass-main}
{
F=\mathsf A\,\Area+\mathsf B\,J
  +\mathsf C\,\cB_2+\mathsf D\,\cK_3,
}
\end{equation}
where ${\mathsf A,\mathsf B,\mathsf C,\mathsf D}$ are constant even scalars,
$\Area$ is the ordinary hyperbolic area of the body triangle,
$J$ is its degree-two fermionic first variation, $\cB_2$ is another
degree-two additive invariant, and $\cK_3$ is a degree-three invariant.
The degree-one and degree-four sectors vanish, so the scalar invariant
space of spin supertriangles has superdimension $3|1$.  Moreover,
\[
\iota^*\Area=\Area,\qquad
\iota^*J=J,\qquad
\iota^*\cB_2=\cB_2,\qquad
\iota^*\cK_3=-\cK_3,
\]
so that the scalar invariants of ordinary supertriangles are exactly
the span of $\Area$, $J$, and $\cB_2$.  We also show that $J$ is precisely
the intrinsic fermionic correction
$
J=D_{\rm super}-D_{\rm body}
$
which was computed in
\cite{defect} in the common-first gauge.

The proof has two main ingredients.  First, we determine explicitly how the
common-first coordinates of the two daughter triangles transform when a side
is subdivided.  Second, these transport formulas convert additivity into
functional equations for the coefficient functions in each fermionic degree.
With our approach of computing degree by degree and solving the resulting functional equations, the 
proof reduces to involved but elementary manipulations in each degree
of ordinary body variables using Newton’s calculus.

Degree zero reduces to the classical valuation theorem; degree one vanishes;
degree two splits into a projected $\beta\beta$ sector and a
$\lambda\beta$ kernel, respectively producing $\cB_2$ and $J$; degree three has a
one-dimensional cubic projection whose unique additive completion is
$\cK_3$; and degree four again vanishes.  The only analytic hypotheses
are
one-sided regularity conditions at degenerate side faces stated in the
corresponding degree arguments,
ranging from ${\mathcal C}^0$ for degrees zero and four, to ${\mathcal C}^1$ for degrees one and three,
and ${\mathcal C}^2$ for degree two. 

The present result concerns scalar triangle valuations, not the full
two-dimensional super scissors-congruence problem.  Scalar invariants do not
separate all spin supertriangles, so further, potentially module-valued,
invariants should enter the general theory.  We shall return to this broader
two-dimensional problem in a sequel.  Beyond it lie the analogous questions
for super hyperbolic tetrahedra and for scissors-congruence groups of higher
dimensional super polyhedra, in direct analogy with the classical questions
originating in Hilbert's third problem
\cite{Hilbert1900,Dehn1900,Sydler1965,Zakharevich2016}.
We hope that the concrete two-dimensional results obtained here might
entice others to study these questions.

\subsubsection*{Acknowledgements and usage of AI} We thank T. Damour, V. Rubtsov, G. Smith and A. M. Zeitlin for discussions.

\noindent
The authors utilized AI agents (ChatGPT) during the project as proof assistants to support computational routines and to discuss mathematical steps. However, all pivotal ideas and approaches were made by the authors and shaped during their discussions. AI was also employed for literature review, image recognition (transferring blackboard formulas and figures into LaTeX), and code verification. After using these services, the authors thoroughly reviewed and edited the content, and they take full responsibility for the final published article.

\section{Background and notation}

We recall only the super-hyperbolic geometry needed for the classification.
Our conventions follow \cite{defect,3d}.  In particular, we use the
super-Minkowski model of the hyperbolic plane, the natural
$\operatorname{OSp}(1|2)$ action, the superdistance and supergeodesic
parametrization, and the central fermionic involution.  The latter will be
responsible for the distinction between spin supertriangles and ordinary
supertriangles.

The {\it body} of a supernumber $x\in\hat{\mathbb R}=\hat{\mathbb R}[0]\otimes\hat{\mathbb R}[1]$ is denoted $x_\#\in{\mathbb R}$.  
An element $g$ in the {\it orthosymplectic group} $\osp$ can be represented by the 3-by-3 matrix $g=\begin{psmallmatrix}
a&b&\alpha\\c&d&\beta\\
\gamma&\delta&f\\\end{psmallmatrix}$, where $a,b,c,d,f\in \hat{\mathbb R}[0]$ are even and $\alpha,\beta,\gamma,\delta\in \hat{\mathbb R}[1]$ are odd, satisfying
\begin{eqnarray}\nonumber
&&\alpha=b\gamma-a\delta, \quad \beta =d\gamma-c\delta, \quad \hskip 2.8ex f=1+\alpha\beta,\nonumber\\
&&\gamma=a\beta-c\alpha, \quad \delta=b\beta-d\alpha, \quad f^{-1}=ad-bc,\nonumber
\end{eqnarray}
and matrix multiplication is defined by
$$\biggl ( \begin{smallmatrix}
a_1&b_1&\alpha_1\\c_1&d_1&\beta_1\\\gamma_1&\delta_1&f_1\\
\end{smallmatrix}\biggr )
\biggl ( \begin{smallmatrix}
a_2&b_2&\alpha_2\\c_2&d_2&\beta_2\\\gamma_2&\delta_2&f_2\\
\end{smallmatrix}\biggr )=
\biggl (\begin{smallmatrix}
a_1a_2+b_1c_2-\alpha_1\gamma_2&a_1b_2+b_1d_2-\alpha_1\delta_2&a_1\alpha_2+b_1\beta_2+\alpha_1 f_2\\
c_1a_2+d_1c_2-\beta_1\gamma_2&c_1b_2+d_1d_2-\beta_1\delta_2&c_1\alpha_2+d_1\beta_2+\beta_1 f_2\\
\gamma_1 a_2+\delta_1 c_2+\delta_1\gamma_2&\gamma_1 b_2+\delta_1 d_2+f_1 \delta_2&-\gamma_1\alpha_2-\delta_1\beta_2+f_1f_2\\
\end{smallmatrix}\biggr ).$$
The unusual signs arise from our choice of blocking the even and odd matrix entries.

The {\it super Minkowski pairing} arising from the Killing form on $\osp$ is denoted \(\pair{\cdot}{\cdot}\), where 
$$\langle X,X\rangle=x_1x_2-x^2+2\phi\psi,$$
for $X=(x_1,x_2,x~|~\phi,\psi)$ in super Minkowski space  
$$\R^{2,1|2}=\{(x_1,x_2,x~|~\phi,\psi):x_1,x_2,x\in \hat{\mathbb R}[0], \phi,\psi\in \hat{\mathbb R}[1]\}.$$  The action of \(g\in\osp\) on $X\in\R^{2,1|2}$ written as $X=\begin{psmallmatrix} ~~x_1&~~x&~\phi\\~~x&~x_2&~\psi\\-\phi&-\psi&~0\\
\end{psmallmatrix}$ via $X\mapsto g^{st} X g$
preserves this pairing, where $g^{st}$ denotes the supertranspose.
The subgroup 
$$\widehat{\rm SL}(2,{\mathbb R})=
\biggl \{\begin{pmatrix}a&b&0\\c&d&0\\0&0&1\\\end{pmatrix}\in{\rm OSp}(1{\rm |}2):
a,b,c,d\in\hat{\mathbb R}[0]\biggr\}$$
is intermediary between the classical ${\rm SL}(2,{\mathbb R})$
and the full ${\rm OSp}(1{\rm |}2)$.  
The nontrivial central lift of the
identity Lorentz transformation acts trivially on the even coordinates and
reverses both fermionic coordinates.  We denote the induced involution by
\[
\iota(x_1,x_2,x\mid\phi,\psi)
=(x_1,x_2,x\mid-\phi,-\psi).
\]
It is precisely the deck transformation which exchanges the two spin choices
used below.  

Another useful subspace, which is not a subgroup of
${\rm OSp}(1|2)$, is defined by setting $$u(\xi,\eta)=\begin{pmatrix}1-{{\xi\eta}\over 2}&0&\xi\\0&1-{{\xi\eta}\over 2}&\eta\\\eta&-\xi&1+\xi\eta\\\end{pmatrix}\in{\rm OSp}(1|2),
~{\rm for}~\xi,\eta\in\hat{\mathbb R}[1].$$

\begin{lemma}\label{factor}
Any element $g\in{\rm OSp}(1|2)$ can be written uniquely as a product
$$g=\begin{psmallmatrix}a&b&0\\c&d&0\\0&0&1\\\end{psmallmatrix}\,u(\xi,\eta)=u(a\xi+b\eta,c\xi+d\eta)\, \begin{psmallmatrix}a&b&0\\c&d&0\\0&0&1\\\end{psmallmatrix},$$
for some $\begin{psmallmatrix}a&b&0\\c&d&0\\0&0&1\\\end{psmallmatrix}\in \widehat{\rm SL}(2,{\mathbb R})$ and fermions $\xi,\eta\in\hat{\mathbb R}[1]$.
%where
%$$u(\alpha,\beta)=\begin{pmatrix}1-{{\alpha\beta}\over 2}&0&\alpha\\0&1-{{\alpha\beta}\over 2}&\beta\\\beta&-\alpha&1+\alpha\beta\\\end{pmatrix}\in{\rm OSp}_{\mathbb C}(1|2).$$\hfill\qedsymbol
\end{lemma}

\begin{proof} This purely algebraic proof is left as an elementary exercise.
\end{proof}

The {\it super hyperbolic plane}
$$\hyperp=\{X\in \R^{2,1|2}:\langle X,X\rangle=1~{\rm and}~(x_1+x_2)_\#\geq 0\}$$ and the {\it (positive) super light cone}
 $$\LLp=\{X\in \R^{2,1|2}:\langle X,X\rangle=0~{\rm and}~(x_1+x_2)_\#\geq 0\}$$
 are each invariant under the action of $\osp$. We shall also parenthetically require the \emph{(positive) special light cone}
\[
\mathcal L_0^+
=
\operatorname{OSp}(1|2)\cdot(1,0,0\mid0,0)
\subset\mathcal L^+.
\]
only in a final example.
Unlike the ordinary light cone, the positive super light cone is
not a single $\operatorname{OSp}(1|2)$ orbit; see \cite{PZ,IPZ} for the relevance to super Teichm\"uller theory.  

The full $\osp$-invariant superdistance $\widehat d$ on $\hyperp$ satisfies
\[
  \cosh \widehat d(P,Q)=\pair{P}{Q}.
\]
Its body is the ordinary hyperbolic distance
\[
  d(P,Q)=d(P_\#,Q_\#)=\bigl(\widehat d(P,Q)\bigr)_\#.
\]
The ordinary law of cosines was shown to hold verbatim for the full even superdistances in $\hyperp$.
Earlier work \cite{defect} relied on a two-form on $\hyperp$ with primitive
$
\frac{y(1+\varphi\psi)}{2}\,{\rm d}\log\left(\frac{x_1}{x_2}\right),
$
which is not invariant under the odd part of $\osp$ and is not used
in the present paper.  The Area considered here is the usual
hyperbolic area of the underlying body with primitive $\frac{y}{2}\,{\rm d}\log\left(\frac{x_1}{x_2}\right)$.

A supergeodesic in $\hyperp$ is determined by two points $U,V\in\LLp$ with $\langle U,V\rangle=2$ and has the even-parameter form
\[
X(\widehat s)=\frac12\left(e^{\widehat s}U+e^{-\widehat s}V\right),
\qquad \widehat s\in\hat{\mathbb R}[0].
\]
Its body
$
s=(\widehat s)_\#\in\mathbb R
$
is ordinary signed arc length along the reduced geodesic, and
\[
\langle X(\widehat s),X(\widehat t)\rangle
=\cosh(\widehat s-\widehat t).
\]
For a finite geodesic segment from $P$ to $Q$ of full even length
$
\widehat c=c+\Delta c
$ with $c=(\hat c)_\#$,
we choose the origin so that $P=X(0)$ and $Q=X(\widehat c)$ and use the canonical affine real parameter $t\in[0,1]$,
where
$\widehat s(t)=t\,\widehat c$.
The body arc length is therefore $s=tc$, while the nilpotent part of the full parameter is $t\,\Delta c$.  In particular, a real body cut position $s\in[0,c]$ determines the corresponding full cut point on the segment.  The full superdistance may therefore have a nilpotent even part, whereas its body is the ordinary real hyperbolic distance.  Henceforth, an un-hatted distance $d$ denotes ordinary distance between the bodies, while a hatted distance denotes the corresponding full even superdistance.

\section{The common-first gauge and moduli of spin supertriangles}
\label{sec:common}

The common-first gauge is the coordinate system used throughout the paper.
Its construction removes two odd degrees of freedom in
$\operatorname{OSp}(1|2)$ by making the first fermionic coordinates of the
three vertices equal.  What remains is a simultaneous sign ambiguity coming
from the central spin involution.  Fixing that sign produces four signed odd
coordinates.  We first derive the common-first gauge and its subdivision law,
and then prove that these four fermions together with the three even pairwise inner products of vertices
form a complete $3|4$ supertriangle modulus.
 
\subsection{Common-first triangles and side-invariants}
\label{sec:common-first}
We begin with the odd normal form.  The common-first fermion will later play
two distinct roles: it is one of the four odd moduli of the triangle, and it
obeys a particularly simple law under subdivision of a distinguished side.

Suppose that
\[
P=(p_1,p_2,p ~|~\alpha,\beta),\quad
Q=(q_1,q_2,q ~|~\gamma,\delta),\quad
R=(r_1,r_2,r ~|~\epsilon,\phi)
\]
span a non-degenerate supertriangle in \(\hyperp\), i,.e.,\, the underlying
bodies $P_\#,Q_\#,R_\#$ are noncollinear in the hyperbrolic plane.

\begin{lemma}[Common-first gauge lemma]
{\rm \cite[Lemma 6]{defect} ~\cite[Lemma 5.1]{3d}}
\label{lemma:corrected-common-first-fermion}\\
For a generic triple \(P,Q,R\) in \(\hyperp\), there exist unique odd
parameters \(\xi,\eta\) such that the pure odd transformation
\(u(\xi,\eta)\) from Lemma \ref{factor} makes the first fermionic coordinates of the three
vertices agree.  The resulting common-first fermionic coordinate is
\[\begin{aligned}
\lambda(P,Q,R)&=
p_1\xi+p\eta+\alpha
+
{3\over2}{\alpha\gamma\epsilon\over\det}\\
&=
q_1\xi+q\eta+\gamma
+
{3\over2}{\alpha\gamma\epsilon\over\det}\\
&=
r_1\xi+r\eta+\epsilon
+
{3\over2}{\alpha\gamma\epsilon\over\det},
\end{aligned}
\]
where
\[
\begin{pmatrix}
\xi\\
\eta
\end{pmatrix}
=
{1\over \det}
\begin{pmatrix}
q-r & q-p\\
r_1-q_1 & p_1-q_1
\end{pmatrix}
\begin{pmatrix}
\gamma-\alpha\\
\epsilon-\gamma
\end{pmatrix}
\]
and
\[
\det=p_1(q-r)+q_1(r-p)+r_1(p-q).
\]
Furthermore, the respective second fermions are given by
\[
\begin{aligned}
\mu_P
&=
p\xi+p_2\eta+\beta+{3\over2}\beta\xi\eta,\\
\mu_Q
&=
q\xi+q_2\eta+\delta+{3\over2}\delta\xi\eta,\\
\mu_R
&=
r\xi+r_2\eta+\phi+{3\over2}\phi\xi\eta.
\end{aligned}
\]
\end{lemma}

\begin{proof}
Given odd parameters \(\xi,\eta\) in the notation of
Lemma~\ref{factor}, write
\[
u(\xi,\eta)\cdot P
=
(p_1',p_2',p'\mid\alpha',\beta'),
\]
and denote the corresponding first fermionic coordinates of
\(u(\xi,\eta)\cdot Q\) and \(u(\xi,\eta)\cdot R\) by
\(\gamma'\) and \(\epsilon'\), respectively.  One computes directly that
\[
\begin{aligned}
\alpha'
&=
p_1\xi+p\eta+\alpha
\left(1+{3\over2}\xi\eta\right),\\
\gamma'
&=
q_1\xi+q\eta+\gamma
\left(1+{3\over2}\xi\eta\right),\\
\epsilon'
&=
r_1\xi+r\eta+\epsilon
\left(1+{3\over2}\xi\eta\right).
\end{aligned}
\]

The conditions \(\alpha'=\gamma'\) and
\(\alpha'=\epsilon'\) are therefore equivalent to
\[
\begin{pmatrix}
p_1-q_1 & p-q\\
p_1-r_1 & p-r
\end{pmatrix}
\begin{pmatrix}
\xi\\
\eta
\end{pmatrix}
=
\left(1+{3\over2}\xi\eta\right)
\begin{pmatrix}
\gamma-\alpha\\
\epsilon-\alpha
\end{pmatrix}.
\]
Let
\[
\begin{pmatrix}
\xi_0\\
\eta_0
\end{pmatrix}
=
\begin{pmatrix}
p_1-q_1 & p-q\\
p_1-r_1 & p-r
\end{pmatrix}^{-1}
\begin{pmatrix}
\gamma-\alpha\\
\epsilon-\alpha
\end{pmatrix}.
\]
Since \(\xi_0\) and \(\eta_0\) are linear combinations of the two odd
elements \(\gamma-\alpha\) and \(\epsilon-\alpha\), one has
\[
\xi_0\eta_0\xi_0
=
\xi_0\eta_0\eta_0
=
0,
\]
and hence
\[
\left(1+{3\over2}\xi_0\eta_0\right)\xi_0=\xi_0,
\qquad
\left(1+{3\over2}\xi_0\eta_0\right)\eta_0=\eta_0.
\]
It follows that the preceding nonlinear system has precisely the same
solution as the linear system
\[
\begin{pmatrix}
p_1-q_1 & p-q\\
p_1-r_1 & p-r
\end{pmatrix}
\begin{pmatrix}
\xi\\
\eta
\end{pmatrix}
=
\begin{pmatrix}
\gamma-\alpha\\
\epsilon-\alpha
\end{pmatrix}.
\]

Its solution is
\[
\begin{aligned}
\xi
&=
{1\over\det}
\left[
(p-r)(\gamma-\alpha)
+
(q-p)(\epsilon-\alpha)
\right],\\
\eta
&=
{1\over\det}
\left[
(r_1-p_1)(\gamma-\alpha)
+
(p_1-q_1)(\epsilon-\alpha)
\right],
\end{aligned}
\]
where
\[
\begin{aligned}
\det
&=
(p-r)(p_1-q_1)
-
(p-q)(p_1-r_1)\\
&=
(qp_1-pq_1)
+
(rq_1-qr_1)
+
(pr_1-rp_1)
\end{aligned}
\]
has nonzero body by genericity.  Thus, the odd parameters
\(\xi,\eta\) and hence the pure odd normalizing transformation
\(u(\xi,\eta)\) as well are uniquely determined.
The displayed formulas for
\(\lambda,\mu_P,\mu_Q,\mu_R\) now follow by direct substitution using
$\xi\eta\,\det
=
(\gamma-\alpha)(\epsilon-\alpha).
$
\end{proof}

\begin{remark}[Odd uniqueness and the residual spin freedom]
\label{rem}
{\rm The uniqueness established in the preceding proof is the uniqueness of the
odd normalizing transformation $u(\xi,\eta)$ once the even frame has
been fixed. Indeed, by Lemma~\ref{factor}, a generic element of
$\osp$ may be written in the form
$$
g=A\circ u(\xi,\eta),
$$
where $A$ belongs to the even subgroup. The common-first equations
determine $\xi$ and $\eta$ uniquely, but they do not determine $A$.
Thus, the common-first condition removes the odd freedom while leaving the
ordinary even Lorentz freedom.

\bigskip

There is a further point which is essential for the interpretation of the upcoming
classification theorem as follows.  The even subgroup of $\osp$ is the spin double cover
\[
\widehat{\mathrm{SL}}(2,\mathbb R)
\longrightarrow
\mathrm{SO}^{+}(2,1)\,,
\]
so an oriented and time-oriented Lorentz transformation of the body
vectors has two spin lifts.  They induce the same transformation on the even
vectors but differ by the central element which reverses all spinor
coordinates simultaneously.  Thus,
\[
(\lambda,\mu_P,\mu_Q,\mu_R)
\longmapsto
(-\lambda,-\mu_P,-\mu_Q,-\mu_R)
\]
is the
involution which exchanges the two spin choices.  Once one choice is retained,
the four common-first fermions are genuine signed coordinates.}
\end{remark}

\begin{definition}[Spin supertriangle]\label{def:spin-supertriangle}
{\rm A \emph{spin supertriangle} is an oriented and time-oriented supertriangle
together with a fixed lift to \(\widehat{\mathrm{SL}}(2,\mathbb R)\) of the
oriented Lorentz frame determined by its body.  On the generic common-first
stratum this is equivalently a choice of the simultaneous sign of
\(\lambda,\mu_P,\mu_Q,\mu_R\).  Forgetting this choice identifies the two
signed coordinate systems by the involution
\[
\iota:(\lambda,\mu_P,\mu_Q,\mu_R)
\longmapsto
(-\lambda,-\mu_P,-\mu_Q,-\mu_R).
\]
All common-first transports between a parent triangle and its subtriangles are
henceforth taken with the spin lifts induced from the parent.}
\end{definition}

\begin{corollary}
For a spin supertriangle, the signed common-first fermion
\[
\lambda(P,Q,R)
=
\lambda(Q,R,P)
=
\lambda(R,P,Q)
\]
is cyclic.  After forgetting the spin choice, this identity is defined only
modulo the common overall sign of
$\lambda,\mu_P,\mu_Q,\mu_R$.
\end{corollary}

\begin{proof}
The linear part is cyclic by the same calculation as before.  Namely,
\(\lambda(P,Q,R)\) leads to the linear system
\[
\begin{pmatrix}
p_1-q_1 & p-q\\
q_1-r_1 & q-r
\end{pmatrix}
\begin{pmatrix}
\xi\\
\eta
\end{pmatrix}
=
\begin{pmatrix}
\gamma-\alpha\\
\epsilon-\gamma
\end{pmatrix},
\]
whereas \(\lambda(Q,R,P)\) leads to
\[
\begin{pmatrix}
q_1-r_1 & q-r\\
r_1-p_1 & r-p
\end{pmatrix}
\begin{pmatrix}
\xi\\
\eta
\end{pmatrix}
=
\begin{pmatrix}
\epsilon-\gamma\\
\alpha-\epsilon
\end{pmatrix}.
\]
These two systems are related via left multiplication by
\[
\begin{pmatrix}
~~0 & ~~1\\
-1 & -1
\end{pmatrix}.
\]
Hence they have the same solution \((\xi,\eta)\).  Moreover,
$
\det
=
p_1(q-r)+q_1(r-p)+r_1(p-q)
$
is cyclically invariant, while
$
\gamma\epsilon\alpha
=
\alpha\gamma\epsilon
$
because a cyclic permutation of three odd variables is an even
permutation.  Thus, the cubic correction is cyclic as well.

The signed common-first fermion is therefore cyclic once the spin lift
has been fixed.  If the spin lift is not fixed, the only remaining
ambiguity is the simultaneous sign reversal given by
$
(\lambda,\mu_P,\mu_Q,\mu_R)
\longmapsto
(-\lambda,-\mu_P,-\mu_Q,-\mu_R)
$,
so the cyclicity statement holds modulo this common overall sign.
\end{proof}

Consider an ordered non-degenerate hyperbolic supertriangle in the notation above.

\begin{lemma}\label{lem:bodyreduction}
Suppose the first fermionic coordinates of the points under consideration are
all scalar multiples of a single odd generator $\rho$.  Suppose further that
the even coordinates differ from their bodies only by nilpotent even corrections
containing a factor $\rho$.  Then in evaluating $\lambda(P,Q,R)$, all
these nilpotent even corrections may be discarded.
\end{lemma}

\begin{proof}
The numerator of the linear term in \(\lambda(P,Q,R)\) is linear in the
first fermions and hence is divisible by \(\rho\), so multiplying
it by any nilpotent even correction divisible by \(\rho\) gives zero.
Similarly, write
\[
        \det=D_0+D_1,
\]
where \(D_0\) is the invertible body of \(\det\), and \(D_1\) is nilpotent
and divisible by \(\rho\). Thus, \(D_1^2=0\), and one has
\[
        \frac{N}{D_0+D_1}
        =
        \frac{N}{D_0}
        -
        \frac{ND_1}{D_0^2}.
\]
The second term vanishes because both \(N\) and \(D_1\) are divisible by
\(\rho\).
Furthermore, the cubic correction in \(\lambda(P,Q,R)\) vanishes whenever the
three first fermions are scalar multiples of the same odd generator
\(\rho\)., and hence only the body even coordinates contribute.
\end{proof}

\begin{lemma}\label{lem:onedim}
Let \(P,Q,R\) have bodies
\[
P_\#=(e^s,e^{-s},0),\qquad
Q_\#=(e^t,e^{-t},0),\qquad
R_\#=(r_1,r_2,r),
\]
with \(s\ne t\), and let their first fermionic coordinates
\(\theta_P,\theta_Q,\theta_R\) all be scalar multiples of one odd
generator. After fixing the spin lift, the signed common-first
fermion is
\begin{equation}\label{eq:onedim}
\lambda(P,Q,R)
=
\frac{e^t\theta_P-e^s\theta_Q}{e^t-e^s}.
\end{equation}
Without a fixed spin lift, this formula is understood up to the common
overall sign of the fermionic coordinates.
\end{lemma}

\begin{proof}
By Lemma~\ref{lem:bodyreduction}, we may use the displayed body even
coordinates.  The determinant is
\[
e^s(0-r)+e^t(r-0)+r_1(0-0)
=
r(e^t-e^s),
\]
and the linear numerator is
\[
\theta_P(e^t r)
-
\theta_Q(e^s r)
+
\theta_R(e^s0-0e^t)
=
r(e^t\theta_P-e^s\theta_Q).
\]
The cubic term is zero because the three first fermions are scalar
multiples of the same odd generator.  Since \(s\ne t\), the denominator
\(e^t-e^s\) is nonzero, and \eqref{eq:onedim} follows.
\end{proof}

\begin{proposition}[Side-additivity of the common-first fermion]
\label{prop:sideadditivity}
Fix a spin structure and use the induced compatible spin lifts for the
triangles \((P,Q,R)\), \((P,T,R)\), and \((T,Q,R)\).  Let \(T\in PQ\),
with
\[
x=d(P_\#,T_\#),\qquad y=d(T_\#,Q_\#),\qquad c=d(P_\#,Q_\#)=x+y.
\]
Then
\[
\tanh(c/2)\,\lam(P,Q,R)
=
\tanh(x/2)\,\lam(P,T,R)
+
\tanh(y/2)\,\lam(T,Q,R).
\]
After forgetting the spin choice, the identity is understood modulo the
simultaneous overall sign of all the fermionic coordinates.
\end{proposition}

\begin{proof}
After fixing the spin lift, the signed common-first fermion is
invariant under spin-preserving \(\osp\)-congruence, so we may choose a representative in which the bodies of the
side \(PQ\) are in normalized position,
\[
        P_\#=(1,1,0),\qquad
       Q_\#=(e^c,e^{-c},0),
\]
and the parent triangle is in common-first gauge with first fermion $\lambda$.
%\[ \alpha_P=\alpha_Q=\alpha_R=\lam.\]
This only
fixes the bodies; the full superpoints may have nilpotent even corrections which
are harmless by Lemma \ref{lem:bodyreduction}.

On the body geodesic from \(P\) to \(Q\), the cut point has body
$
        T_\#=(e^x,e^{-x},0)
$,
and the body interpolation coefficients are
\[
        S=\frac{\sinh y}{\sinh c},\qquad
        U=\frac{\sinh x}{\sinh c}.
\]
Since the endpoint first fermions of \(P\) and \(Q\) are both \(\lam\), the first
fermion of the cut point is
\begin{equation}\label{eq:sigma}
        \sigma=(S+U)\lam
        =\frac{\sinh y+\sinh x}{\sinh c}\,\lam.
\end{equation}
Nilpotent corrections to the interpolation coefficients contain a factor
\(\lam\), and therefore disappear after multiplication by the endpoint first
fermion \(\lam\).

Apply Lemma \ref{lem:onedim} first to \((P,T,R)\), where \((s,t)=(0,x)\),
\(\theta_P=\lam\), and \(\theta_T=\sigma\), giving
\begin{equation}\label{eq:leftraw}
        \lam(P,T,R)=\frac{e^x\lam-\sigma}{e^x-1}.
\end{equation}
Using \eqref{eq:sigma} and \(c=x+y\), this simplifies to
\begin{equation}\label{eq:leftfactor}
        \lam(P,T,R)
        =
        \frac{e^{y/2}\cosh(x/2)}{\cosh(c/2)}\,\lam.
\end{equation}
Apply Lemma \ref{lem:onedim} to \((T,Q,R)\), where \((s,t)=(x,c)\),
\(\theta_T=\sigma\), and \(\theta_Q=\lam\).  This gives
\begin{equation}\label{eq:rightraw}
        \lam(T,Q,R)=\frac{e^c\sigma-e^x\lam}{e^c-e^x},
\end{equation}
and hence
\begin{equation}\label{eq:rightfactor}
        \lam(T,Q,R)
        =
        \frac{e^{-x/2}\cosh(y/2)}{\cosh(c/2)}\,\lam.
\end{equation}
Finally, multiply \eqref{eq:leftfactor} by \(\tanh(x/2)\) and
\eqref{eq:rightfactor} by \(\tanh(y/2)\).  Their sum is
\[
\frac{e^{y/2}\sinh(x/2)+e^{-x/2}\sinh(y/2)}{\cosh(c/2)}\,\lam.
\]
Since \(c=x+y\),
we have
        $e^{y/2}\sinh(x/2)+e^{-x/2}\sinh(y/2)=\sinh(c/2)$, and
the sum is therefore
\[
        \frac{\sinh(c/2)}{\cosh(c/2)}\,\lam
        =\tanh(c/2)\,\lam.
\]
Because the parent triangle is in common-first gauge, the signed
invariant \(\lam(P,Q,R)\), relative to the chosen spin lift, is
represented by the common first fermion \(\lam\).  Moreover, the cubic
term is
$
{3\over2}{\lam^3\over\det}=0
$,
and the asserted identity follows.  Reversing the spin lift changes
all three common-first fermions by the same overall sign and therefore
preserves the identity.\end{proof}

This does {\sl not define a genuinely additive invariant of spin supertriangles},
because the coefficient attached to a side depends on its two endpoints and is
therefore not preserved when that side is subdivided. We instead regard it as a
{\it side-invariant} of a spin supertriangle with a distinguished oriented edge.
Such side-invariants are comparatively easy to construct, unlike genuine
additive invariants. The one obtained here is nevertheless especially elegant for its law of subdivision
in Proposition
\ref{prop:sideadditivity}, which is one reason why we
highlight it here.  This side-invariant will moreover be useful in the sequel.

\subsection{Complete modulus of spin supertriangles}
We now show that the common-first data are not merely convenient coordinates:
they can be extended to a complete modulus.  For this purpose, write the three vertices as
\(P_1,P_2,P_3\) and after imposing the common-first gauge, put
\[
P_i=(v_i\mid\lambda,\mu_i),
\]
where \(v_i\) denotes its even part.  The super-hyperboloid equation gives
$
\langle v_i,v_i\rangle=1-2\lambda\mu_i
$.
Define
\[
y_i=(1+\lambda\mu_i)v_i,
\]
so that
\[
\langle y_i,y_i\rangle
=
(1+\lambda\mu_i)^2\langle v_i,v_i\rangle
=
(1+2\lambda\mu_i)(1-2\lambda\mu_i)
=
1.
\]
Conversely, since
$
(1-\lambda\mu_i)(1+\lambda\mu_i)=1
$,
we have
$
v_i=(1-\lambda\mu_i)y_i
$,
and so
\begin{equation}\label{eq:common-first-normal-form-main}
P_i
=
\bigl((1-\lambda\mu_i)y_i\mid\lambda,\mu_i\bigr).
\end{equation}

\begin{lemma}[Inner products in the common-first slice]
\label{lem:inner-products-main}
Put
\[
G_{ij}=\langle P_i,P_j\rangle,
\qquad
A_{ij}=\langle y_i,y_j\rangle,
\qquad
i\ne j,
\]
where \(P_i\) is written in the normal form
\eqref{eq:common-first-normal-form-main}.  Then
\begin{equation}\label{eq:Gij-main}
G_{ij}
=
A_{ij}
+
\lambda(\mu_i+\mu_j)(1-A_{ij}),
\end{equation}
and consequently
\begin{equation}\label{eq:Aij-recover-main}
A_{ij}
=
\frac{G_{ij}-\lambda(\mu_i+\mu_j)}
     {1-\lambda(\mu_i+\mu_j)}
=
G_{ij}
+
\lambda(\mu_i+\mu_j)(G_{ij}-1).
\end{equation}
\end{lemma}

\begin{proof}
Using \eqref{eq:common-first-normal-form-main}, we find
\begin{align*}
G_{ij}
&=
(1-\lambda\mu_i)(1-\lambda\mu_j)A_{ij}
+
\lambda\mu_j+\lambda\mu_i\\
&=
A_{ij}
-
\lambda(\mu_i+\mu_j)A_{ij}
+
\lambda(\mu_i+\mu_j)\\
&=
A_{ij}
+
\lambda(\mu_i+\mu_j)(1-A_{ij}).
\end{align*}
Since \(\lambda(\mu_i+\mu_j)\) squares to zero, this solves the affine
relation and gives \eqref{eq:Aij-recover-main}.
\end{proof}

\begin{theorem}[Completeness of the common-first spin modulus]
\label{thm:complete}
\label{lem:common-first-complete-main}
On the generic common-first stratum, after fixing orientation, time
orientation, and the spin choice of Definition~\ref{def:spin-supertriangle},
the data
\begin{equation}\label{eq:common-first-package-main}
\bigl(
\langle P_1,P_2\rangle,
\langle P_2,P_3\rangle,
\langle P_3,P_1\rangle;
\lambda,
\mu_1,
\mu_2,
\mu_3
\bigr)
\end{equation}
form a complete invariant of ordered spin supertriangles under
spin-preserving \(\osp\)-congruence.  Forgetting the spin choice gives the
ordinary supertriangle modulus by the identification
\[
(\lambda,\mu_1,\mu_2,\mu_3)
\sim
(-\lambda,-\mu_1,-\mu_2,-\mu_3).
\]
\end{theorem}

\begin{proof}
Put two generic ordered spin supertriangles into common-first normal form using their chosen compatible lifts.  Suppose
their data \eqref{eq:common-first-package-main} agree.  Their signed
fermionic coordinates
$
\lambda,\mu_1,\mu_2,\mu_3
$
therefore agree, and Lemma~\ref{lem:inner-products-main} recovers the ordinary
Lorentzian Gram entries
$
A_{ij}=\langle y_i,y_j\rangle
$
from
$
G_{ij}=\langle P_i,P_j\rangle
$
and the fermionic coordinates.  Hence the two ordered triples
$
(y_1,y_2,y_3)
\text{~and~}
(y_1',y_2',y_3')
$
have the same ordinary Lorentzian Gram matrix.

By Lorentzian Gram rigidity, together with the fixed orientation and time
orientation, there is a unique element
\[
A\in\mathrm{SO}^{+}(2,1)
\]
such that
$
Ay_i=y_i',
{\rm{~for~}}
i=1,2,3
$.
The even subgroup of \(\osp\) is the spin double cover of
\(\mathrm{SO}^{+}(2,1)\), so \(A\) has two lifts differing by the central
spin element.  These two lifts act identically on the even vectors and
reverse all fermionic coordinates simultaneously.  Since the signed
fermionic coordinates of the two common-first triples agree, the fixed spin
structure selects the lift carrying
$
(\lambda,\mu_1,\mu_2,\mu_3)
$
to the corresponding fermionic coordinates of the second triple.

For this lift, one has
\[
(1-\lambda\mu_i)Ay_i
=
(1-\lambda\mu_i)y_i',
\]
and therefore
\[
\bigl((1-\lambda\mu_i)Ay_i\mid\lambda,\mu_i\bigr)
=
\bigl((1-\lambda\mu_i)y_i'\mid\lambda,\mu_i\bigr)
=
P_i'
\]
for \(i=1,2,3\).  Thus, the two common-first representatives are related by
an element of \(\osp\).  Undoing the two common-first normalizations gives a spin-preserving
\(\osp\)-congruence between the original ordered spin supertriangles.  This
proves completeness on the spin refinement.  If the spin choice is forgotten,
the two possible lifts differ exactly by
\[
(\lambda,\mu_1,\mu_2,\mu_3)
\longmapsto
(-\lambda,-\mu_1,-\mu_2,-\mu_3),
\]
which gives the stated quotient description of ordinary supertriangles.
\end{proof}

\section{The subdivision calculus in common-first coordinates}
\label{sec:subdivision-calculus}

The classification rests on comparing a parent spin supertriangle with the
two spin supertriangles obtained by cutting one of its sides.  Each daughter
has its own common-first gauge, so before their coefficients can be added they
must be transported back to the common-first coordinates of the parent.
This section derives those transport formulas once and for all.

Throughout, the daughter triangles carry the spin lifts induced from the
parent.  Hence all fermionic coordinates below are signed.  The body side
lengths are denoted
\[
a,b,x,y,m,~{\rm and}~c=x+y,
\]
as illustrated in Figure \ref{fig:body}.  These are ordinary real body
lengths.  By a side subdivision \(T\in PQ\) we mean the full cut point determined
by the canonical affine parametrization above; equivalently, the cut is specified
by its body point \(T_\#\in P_\#Q_\#\), or by ordinary real body arc length along
the reduced geodesic.  The corresponding full even superdistances may contain
nilpotent corrections.  These corrections are
accounted for explicitly by the common-first transport and by the positive
fermionic-degree terms.  Coefficient functions are always evaluated on the
body lengths.  This convention will be used throughout the classification.
The ordinary hyperbolic sine and cosine laws will be used freely for the three body triangles,
and we note parenthetically that the identical formulas hold in super hyperbolic geometry \cite{defect,3d}.

\begin{figure}[H]
\centering
\includegraphics[
  width=0.75\textwidth,
  trim={1cm 14.5cm 1cm 4cm},
  clip
]{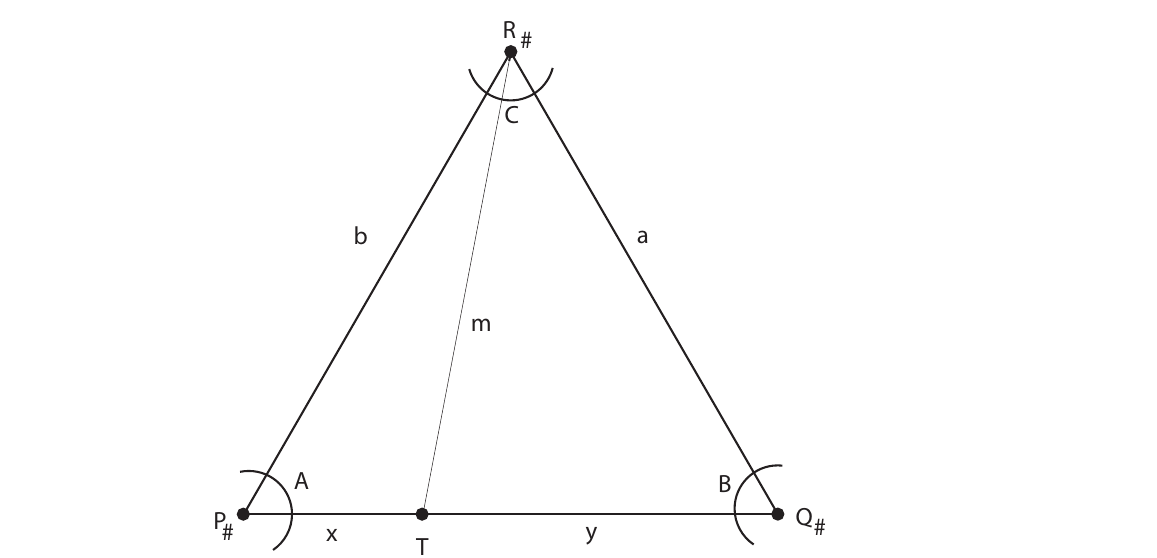}
\caption{Notation for body triangle: hyperbolic lengths and angles.}
\label{fig:body}
\end{figure}

We define
$c=x+y$
and introduce the expressions
$$S={{{\rm sinh}\,y}\over{{\rm sinh}\,c}}~~{\rm and}~~U={{{\rm sinh}\,x}\over{{\rm sinh}\,c}}.$$
The common first fermions on the left ($g$ for gauche) and right ($d$ for droit) are respectively given by
$$\lambda^\ell=\lambda(P,T,R)=g\lambda~~{\rm and}~~\lambda ^r=\lambda(Q,R,T)= d\lambda\,,
~{\rm where}~~\lambda=\lambda(P,Q,R),$$
according to Proposition \ref{prop:sideadditivity} with
$$g=e^{y\over 2}~ {{{\rm cosh}\,{x\over 2}}\over{{\rm cosh}\,{c\over 2}}}~~{\rm and}~~
d=e^{-{x\over 2}}~{{{\rm cosh}\,{y\over 2}}\over{{\rm cosh}\,{c\over 2}}}.$$

For the second fermions in the left triangle, we have the formulas
$$\begin{aligned}
\beta_P^\ell&=\beta_P+\ell_P\lambda, ~~\beta^\ell_R=\beta_R+\ell_R\lambda,\\
\beta_T^\ell&=S\beta_P+U\beta_Q+\ell_T\lambda,
\end{aligned}$$
where
$$\begin{aligned}
\ell_P&={{1-G_1}\over G}(g-1),~~\ell_T={{1-G_1}\over G}(g-1)e^{-x},\\
\ell_R&=\biggl [G+G_2{{1-G_1}\over G}\biggr](g-1),
\end{aligned}$$
with
$$\begin{aligned}
G_1&={{e^x\, {\rm cosh}\,b-{\rm cosh}\,m}\over {{\rm sinh}\, x}},~~
G_2={{{\rm cosh}\, m-e^{-x}\,{\rm cosh}\, b}\over{{\rm sinh}\, x}},\\
G&={{\sqrt{1~~+~~2\,{\rm cosh}\,m~{\rm cosh}\, b~{\rm cosh}\, x~~-~~{\rm cosh}^2 m-{\rm cosh}^2 b-{\rm cosh}^2 x}}\over{{\rm sinh}\, x}}.
\end{aligned}$$

For the second fermions in the right triangle, we have the formulas
$$\begin{aligned}
\beta_Q^r&=\beta_Q+r_Q\,\lambda, ~~\beta^r_R=\beta_R+r_R\lambda,\\
\beta_T^r&=S\beta_P+U\beta_Q+r_T\lambda,
\end{aligned}$$
where
$$\begin{aligned}
r_Q&={{1-e^{-y}\,D_1}\over D}(d-1)e^{-y},~~r_T={{1-e^{-y}\,D_1}\over D}(d-1),\\
r_R&=e^{-y}\, D\,(d-1)
+{{D_2\,(d-1)(1-e^{-y}\,D_1)}\over D},
\end{aligned}$$
with
$$\begin{aligned}
D_1&={{e^y\, {\rm cosh}\,m-{\rm cosh}\,a}\over {{\rm sinh}\, y}},~~
D_2={{{\rm cosh}\, a-e^{-y}\,{\rm cosh}\, m}\over{{\rm sinh}\, y}},\\
D&={{\sqrt{1~~+~~2\,{\rm cosh}\,m~{\rm cosh}\, a~{\rm cosh}\, y~~-~~{\rm cosh}^2 m-{\rm cosh}^2 a-{\rm cosh}^2 y}}\over{{\rm sinh}\, y}}.
\end{aligned}$$

\begin{proof}
[Verification of these transport formulas]
Only the last displayed formulas require verification.  In the side-normal
body frame for the left subtriangle, the body coordinates are
\[
  P_\#=(1,1,0),\qquad
  T_\#=(e^x,e^{-x},0),\qquad
  R_\#=(r_1^L,r_2^L,r^L).
\]
The common-first value on the left is \(p\lambda\).  If the odd correction which
re-gauges the left triangle is written as \((\xi_L,\eta_L)\), then the first
fermions satisfy, after applying Lemma \ref{lem:bodyreduction},
\[
  \lambda+\xi_L=p\lambda,
\]
\[
  (S+U)\lambda+e^x\xi_L=p\lambda,
\]
\[
  \lambda+r_1^L\xi_L+r^L\eta_L=p\lambda.
\]
\((S+U)\lambda\) is the first fermion of the cut point at the relevant
order.  The nilpotent even parts of \(S\) and \(U\)
are already divisible by \(\lambda\), hence disappear after multiplication by
\(\lambda\).  The first equation gives
$
  \xi_L=(p-1)\lambda
$,
and the third gives
$
  \eta_L=\frac{1-r_1^L}{r^L}(p-1)\lambda
$.
Substituting this odd correction into the transformation law for the second
fermion gives the required three equations.
The use of body coordinates is legitimate here because every nilpotent even
correction to a common-first even coordinate is of type \(\lambda\beta\); it vanishes when
multiplied by \(\xi_L\) or \(\eta_L\), because both are multiples of
\(\lambda\).

The right formulas are identical in spirit.  In the right side-normal body
frame,
\[
  T_\#=(1,1,0),\qquad
  Q_\#=(e^y,e^{-y},0),\qquad
  R_\#=(r_1^R,r_2^R,r^R),
\]
and the common-first value is \(q\lambda\).  Solving the three first-fermion
equations gives
\[
  \xi_R=(q-(S+U))\lambda=e^{-y}(q-1)\lambda,
\]
\[
  \eta_R=\frac{(q-1)(1-e^{-y}r_1^R)}{r^R}\lambda.
\]
The second-fermion transformation law gives
the further required equations.  Lemma \ref{lem:bodyreduction}, again shows that no omitted nilpotent even correction
contributes to the displayed coefficients.
\end{proof}

For a parent triangle \(\Delta\) and a subtriangle
\(\Delta'\subset\Delta\), whenever an expression written in the
common-first coordinates of \(\Delta'\) is compared with one on
\(\Delta\), we rewrite it in the common-first coordinates of
\(\Delta\) by the full transport formulas.  We denote this substitution by
$
\mathsf T_{\Delta'\to\Delta}
$.
For nested subtriangles
$
\Delta''\subset\Delta'\subset\Delta
$,
the transport is transitive in the sense that
\begin{equation}\label{eq:transport-transitive}
\mathsf T_{\Delta''\to\Delta}
=
\mathsf T_{\Delta'\to\Delta}\circ
\mathsf T_{\Delta''\to\Delta'}.
\end{equation}
More generally, the resulting substitution is independent of the
sequence of intermediate subdivisions, and the common-first
transport is thus flat in the usual sense that all transport diagrams associated with iterated
subdivision commute.  We shall tacitly apply these identifications in the sequel.

\section{Classification by fermionic degree}
\label{sec:classification-proof}

With the complete common-first modulus and its subdivision transport in hand,
the classification becomes straight-forward since every scalar invariant decomposes into
fermionic degrees $0,\ldots,4$, and additivity may be imposed separately in
each degree.  We first state the result, then treat the five degrees in order.
All triangles below are spin supertriangles and all daughter triangles carry
the compatible spin lifts inherited from their parent.  The parent triangle
$PQR$ is in common-first gauge with fermions $(\lambda,\beta_P,\beta_Q,\beta_R)$
as in Section \ref{sec:common}. The regularity
required at degenerate body faces is stated explicitly in each degree.
Throughout, additivity refers to the body-parameter geodesic subdivision
specified in Figure \ref{fig:body} and Section~\ref{sec:subdivision-calculus}, whose notation 
\(a,b,c=x+y,m\) for body distances between pairs among vertices $P,Q,R,T$ and \(S,U,g,d,\ell_P,\ell_T,\ell_R,r_T,r_Q,r_R\) for the body transport
coefficients we henceforth tacitly adopt.

\begin{theorem}[Classification for spin supertriangles]
\label{thm:main}
Every regular cyclic additive scalar invariant of non-degenerate spin
super-hyperbolic triangles in the signed common-first \(3|4\) modulus can be
written uniquely in the form
\begin{equation}\label{eq:main-theorem-form}
  F=\mathsf A\Area+\mathsf B J+\mathsf C\cB_2+\mathsf D\cK_3,
\end{equation}
where \(\mathsf A,\mathsf B,\mathsf C,\mathsf D\) are constant even scalars
in the ground coefficient algebra, \(\Area\) is the usual hyperbolic area of
the underlying body triangle, and \(J=\delta_{\!f}\Area\) is its degree-two
common-first fermionic variation; intrinsically,
\(J=D_{\rm super}-D_{\rm body}\), and in common-first gauge it agrees with
the degree-two correction computed in \cite{defect}. \(\cB_2\) is another degree-two
generator, and \(\cK_3\) is the degree-three generator, both of which will be
explicated in the proof.  Thus, the space of additive invariant of spin hyperbolic supertrianges
has dimension  \(3|1\), with even part $\operatorname{span}\{\Area,J,\cB_2\}$ and odd part $\operatorname{span}\{\cK_3\}$.
Under the spin involution \(\iota\), we have
\[
\iota^*\Area=\Area,\qquad
\iota^*J=J,\qquad
\iota^*\cB_2=\cB_2,\qquad
\iota^*\cK_3=-\cK_3.
\]
\end{theorem}

\begin{corollary}[Descent to ordinary supertriangles]
\label{cor:ordinary-supertriangles}
The regular cyclic additive scalar invariants of ordinary super-hyperbolic
triangles, with the spin choice forgotten, are precisely
\[
F=\mathsf A\Area+\mathsf B J+\mathsf C\cB_2.
\]
In particular there is no nonzero odd scalar invariant of ordinary
supertriangles.
\end{corollary}

\begin{proof}
Forgetting the spin choice identifies the signed common-first coordinates by
\[
(\lambda,\beta_P,\beta_Q,\beta_R)
\longmapsto
(-\lambda,-\beta_P,-\beta_Q,-\beta_R).
\]
A homogeneous term of fermionic degree \(k\) is therefore multiplied by
\((-1)^k\).  A scalar invariant descends to the ordinary supertriangle
quotient exactly when it is fixed by this involution.  The assertion now
follows from Theorem~\ref{thm:main}: the three even generators are fixed and
the degree-three generator is anti-invariant.
\end{proof}

All four generators $\Area,J,\cB_2$ and $\cK_3$ will be explicitly described in
the ensuing subsections using the notation of 
Section~\ref{sec:subdivision-calculus}.

\subsection{Degree zero}

There are no fermions in degree zero, so the problem reduces immediately to
the classical valuation problem on the body triangle for which continuity at the
degenerate faces is sufficient.

\begin{proposition}[Degree zero]
\label{prop:degree0-main} 
The degree-zero part of a ${\mathcal C}^0$-regular cyclic additive scalar invariant of
hyperbolic spin supertriangles is a scalar multiple of the hyperbolic area
$\Area$.  This generator is spin-independent and therefore also defines the
degree-zero invariant of ordinary supertriangles.
\end{proposition}

\begin{proof}
In degree zero, the invariant $I_0(a,b,c)$ depends only on the body triangle.  Its regular
extension to a collapsed side vanishes automatically: taking the limit of a
side cut in which one daughter collapses gives
\[
I_0(t,t,0)=I_0(t,t,0)+I_0(t,t,0),
\]
and hence \(I_0(t,t,0)=0\).  Additivity under side cuts therefore extends it,
by triangulation, to a continuous isometry-invariant simple valuation of
convex hyperbolic polygons; the value is independent of the triangulation
because two geodesic triangulations admit a common refinement.  The result then
follows from Klain's characterization
\cite{Klain2006HyperbolicValuations}
of hyperbolic area as the unique continuous isometry-invariant simple
valuation on convex polygons in the hyperbolic plane.
\end{proof}

\subsection{Degree one}

Degree one already vanishes after passage to ordinary supertriangles, because
the spin involution changes its sign.  The stronger statement proved here is
that it vanishes on the spin cover as well.  We assume only $C^1$ one-sided
regularity at the degenerate body faces.

\begin{proposition}[No degree-one generator]\label{prop:degree-one-zero}
The only \(C^1\)-regular cyclic additive degree-one scalar invariant of hyperbolic spin supertriangles vanishes.
\end{proposition}

\begin{proof}

Write the putative degree-one invariant as
\[
I_1^{PQR}=\lambda f(a,b,c)
+\beta_P u(a,b,c)
+\beta_Q u(b,c,a)
+\beta_R u(c,a,b).
\]
%Let \(T\in PQ\) with $$ a=d(Q_\#,R_\#),~ b=d(R_\#,P_\#),~c=d(P_\#,Q_\#)$$ setting also $$x=d(P_\#,T_\#),~ y=d(T_\#,Q_\#),~ c=x+y,~ m=d(T_\#,R_\#) $$ for these body distances as in Figure \ref{fig:body}.  As before, \[ S=\frac{{\rm sinh}\,y}{{\rm sinh}\,c},\qquad U=\frac{{\rm sinh}\,x}{{\rm sinh}\,c}.\]
The common-first fermions of the two subtriangles satisfy
\[
\lambda^\ell=g\lambda,\qquad
\lambda^r=d\lambda, 
\]
%where \begin{equation}\label{eq:gd-degree-one} g= e^{y/2}\frac{{\rm cosh}(x/2)}{{\rm cosh}(c/2)}, \qquad d= e^{-x/2}\frac{{\rm cosh}(y/2)}{{\rm cosh}(c/2)}.\end{equation}
and the full transport of the second fermions contains additional terms
proportional to \(\lambda\).  These additional terms affect only the
coefficient of \(\lambda\) in a degree-one expression.  They do not
affect the coefficients of
\(\beta_P,\beta_Q,\beta_R\).  We therefore first compare those three
coefficients.  Using
\[
\beta_T^\ell\equiv\beta_T^r
\equiv S\beta_P+U\beta_Q\pmod{\lambda},
\]
additivity $I_1^{PQR}=I_1^{PTR}+I_1^{QRT}$ gives
\begin{equation}\label{eq:u1-degree-one}
u(a,b,c)
=
u(m,b,x)
+S\,[u(b,x,m)
+u(a,m,y)],
\end{equation}
\begin{equation}\label{eq:u2-degree-one}
u(b,c,a)
=u(m,y,a)+U[
u(b,x,m)
+\,u(a,m,y)],
\end{equation}
\begin{equation}\label{eq:u3-degree-one}
u(c,a,b)
=
u(x,m,b)
+u(y,a,m).
\end{equation}

We first prove that \(u=0\).  To this end, let us begin by
showing that $u$ vanishes on the degenerate locus.  
Sending $T\to P$ gives $x\to 0$, $y\to c$, $m\to b$, hence $S\to 1$ and $U\to 0$.
Thus, (\ref{eq:u3-degree-one}) gives $u(0,t,t)=0$, (\ref{eq:u1-degree-one}) gives $u(t,t,0)=-u(t,0,t)$, and so
\begin{equation}
u(0,t,t)=0,\qquad u(t,t,0)=h(t),\qquad u(t,0,t)=-h(t), \label{tag4}
\end{equation}
where\[
h(t)=u(t,t,0).
\]
(\ref{eq:u1-degree-one}-\ref{eq:u3-degree-one}) respectively give

\begin{equation}
u(y,x,c)=\frac{\sh y}{\sh(x+y)}\bigl[h(x)-h(y)\bigr], \label{tag5}
\end{equation}
\begin{equation}
u(x,c,y)=\frac{\sh x}{\sh(x+y)}\bigl[h(x)-h(y)\bigr], \label{tag6}
\end{equation}
\begin{equation}
u(c,y,x)=h(y)-h(x). \label{tag7}
\end{equation}

We claim that $h=0$.  To see this, consider a degenerate triangle with
\[
PR=p,\qquad RQ=q,\qquad PQ=p+q,
\]
and choose $T$ between $P$ and $R$, with
\[
PT=x,\qquad TR=m,\qquad p=x+m.
\]
For this cut, (\ref{eq:u1-degree-one}) reads
\begin{equation}\label{eq:(*)}
u(q,p,p+q)=u(m,p,x)
 +\frac{\sh(m+q)}{\sh(p+q)}
 \biggl[u(p,x,m)+u(q,m,m+q)\biggr].
\end{equation}
Meanwhile (\ref{tag5}) with $y\to q, x\to p$ gives
\[
u(q,p,p+q)=\frac{\sh q}{\sh (p+q)}\,[h(p)-h(q)],
\]
(\ref{tag6}) with $x\to m, y\to x$ gives
\[
u(m,p,x)=\frac{\sh m}{\sh p}\,[h(m)-h(x)],
\]
and (\ref{tag7}) with $x\to m, y\to x$ gives
\[
u(p,x,m)=[h(x)-h(m)].
\]

Moreover, (\ref{tag5}) with $x\to m, y\to q$ gives
\[
u(q,m,m+q)=\frac{\sh q}{\sh (m+q)}\,[h(m)-h(q)],
\] 
so (\ref{eq:(*)}) becomes
$$\begin{aligned}
\frac{\sh q}{\sh (p+q)}\,[h(p)-h(q)]&=\frac{\sh m}{\sh p}\,[h(m)-h(x)]\\
&+\frac{\sh (m+q)}{\sh (p+q)}\,\biggl [
[h(x)-h(m)]+\frac{\sh q}{\sh(m+q)}[h(m)-h(q)]
\biggr].\\\end{aligned}$$
Multiply through by $\sh p\,\sh(p+q)$ and cancel the like term
$[\sh p\,\sh q]\,h(q)$ to get
$$
\sh p\,\sh q\,[h(p)-h(m)]
=\biggl\{\sh(p+q)\,\sh m-\sh p\,\sh(m+q)\biggr\}
[h(m)-h(x)].
$$
Since $p=m+x$, the expression in curly braces is
$-\sh q\,\sh x$.  Dividing the previous equation by $\sh q$ gives
$$\sh p\,[h(p)-h(m)]=-\sh x\,[h(m)-h(x)].$$
Finally taking $m=x, p=2x$, gives $\sh 2x\,[h(2x)-h(x)]=0$, so we conclude
\[
h(2x)=h(x),
\]
whence
\[
h(x)=h(x/2)=h(x/2^2)=\cdots .
\]
By regularity at the origin and (\ref{tag4}), $h(0)=0$, hence $h\equiv0$, and so
(\ref{tag5}-\ref{tag7}) imply that
\begin{equation}
u=0
\quad\text{on every degenerate triangle}. \label{tag9}
\end{equation}

It remains to propagate this to non-degenerate triangles as follows.
fix an arbitrary non-degenerate triangle and subdivide $PQ$ into $n$ equal pieces
\[
P=T_0,T_1,\ldots,T_n=Q,
\qquad d(T_{i-1},T_i)=\delta=\frac cn.
\]
Set $r_i=d(R,T_i)$.  Repeated application of
\eqref{eq:u3-degree-one} gives the exact identity
\[
u(c,a,b)=\sum_{i=1}^n u(\delta,r_i,r_{i-1}). 
\]
By the triangle inequality,
\[
|r_i-r_{i-1}|\leq\delta,
\]
so every summand approaches a point $(0,r,r)$ as $\delta\to0$.
At $(0,r,r)$ two degenerate faces meet transversely.  Since $u$ vanishes
on both faces by (\ref{tag9}), its differential annihilates the tangent
space of each face.  Those two tangent spaces span the ambient
three-dimensional side-length space, so $du=0$ at their intersection.
The assumed $C^1$ boundary regularity therefore gives
\[
u(\delta,r_i,r_{i-1})=o(\delta)
\]
uniformly for $r_i$ in the compact interval swept out by the fixed side,
whence
\[
|u(c,a,b)|\leq n\,o(c/n)\longrightarrow0.
\]
Thus, $u(c,a,b)=0$ for every triangle, as desired.

Having shown that $u\equiv 0$,
 the invariant reduces to
$
F_1=\lambda f(a,b,c)
$, and it remains to prove the vanishing of $f$.
In particular, the lower \(\lambda\)-corrections occurring in the
full transport of the second fermions no longer contribute.
Since the signed common-first fermion \(\lambda\) is cyclic after
fixing the spin structure, cyclicity of \(F_1\) implies
%\begin{equation}\label{eq:f-cyclic}
$f(a,b,c)=f(b,c,a)=f(c,a,b)$.
%\end{equation}
Additivity across the cut \(T\in PQ\), together with
\(\lambda^\ell=g\lambda\) and \(\lambda^r=d\lambda\), gives
\begin{equation}\label{eq:f-subdivision}
f(a,b,c)
=
g\,f(m,b,x)+d\,f(a,m,y).
\end{equation}
Letting \(x\to0\) gives
$
m\to b$, $ y\to c
$
and
$
g\longrightarrow
\frac{e^{c/2}}{{\rm cosh}(c/2)},
d\longrightarrow1,
$
so \eqref{eq:f-subdivision} becomes
\[
f(a,b,c)
=
\frac{e^{c/2}}{{\rm cosh}(c/2)}\,f(b,b,0)
+
f(a,b,c).
\]
Since
$
\frac{e^{c/2}}{{\rm cosh}(c/2)}\neq0,
$
we obtain
%\begin{equation}\label{eq:f-boundary-one}
$f(t,t,0)=0$,
%\end{equation}
for every \(t>0\), and by cyclicity,
\begin{equation}\label{eq:f-all-boundaries}
f(t,t,0)=f(t,0,t)=f(0,t,t)=0.
\end{equation}

Using dummy variables
\(r,s,q\), consider the partial derivatives
\[
A(t)
=
\left.
\frac{\partial}{\partial r}f(r,s,q)
\right|_{(r,s,q)=(t,t,0)}
\]
and
\[
B(t)
=
\left.
\frac{\partial}{\partial q}f(r,s,q)
\right|_{(r,s,q)=(t,t,0)}
\]
at $(r,s,q)=(t,t,0)$.
Since \(f(t,t,0)=0\), for  \(\delta,\varepsilon\to0\) we have
\begin{equation}\label{eq:f-taylor}
f(t+\delta,t,\varepsilon)
=
A(t)\delta+B(t)\varepsilon
+o(|\delta|+|\varepsilon|).
\end{equation}

Again, fix an arbitrary non-degenerate triangle using the same
arclength parametrization \(X_s\) of \(PQ\), with
$
r(s)=d(X_s,R), r(0)=b, r(c)=a,
$
and define
\[
\psi(s)=f(r(s),b,s).
\]
For \(h>0\) with $
x=s, y=h, m=r(s)
$, apply \eqref{eq:f-subdivision} to the triangle
\((P,X_{s+h},R)\) cut at \(X_s\), so
\begin{equation}\label{eq:psi-subdivision}
\psi(s+h)
=
g_{s,h}\,\psi(s)
+
d_{s,h}\,
f(r(s+h),r(s),h),
\end{equation}
where
\[
g_{s,h}
=
e^{h/2}
\frac{{\rm cosh}(s/2)}
     {{\rm cosh}((s+h)/2)}
\]
and
\[
d_{s,h}
=
e^{-s/2}
\frac{{\rm cosh}(h/2)}
     {{\rm cosh}((s+h)/2)}.
\]
Also, at \(h=0\), we have
$
g_{s,0}=1, d_{s,0}
=
\frac{e^{-s/2}}{{\rm cosh}(s/2)}
=
\frac{2}{1+e^s}
$, and
\[
\left.
\frac{\partial g_{s,h}}{\partial h}
\right|_{h=0}
=
\frac12-\frac12\tanh(s/2)
=
\frac{1}{1+e^s}.
\]
It follows that
\begin{equation}\label{eq:g-small-h}
g_{s,h}
=
1+\frac{h}{1+e^s}+o(h).
\end{equation}

Meanwhile,
$
r(s+h)=r(s)+h r'(s)+o(h)
$,
so applying \eqref{eq:f-taylor} at
$
(t,t,0)=(r(s),r(s),0)
$
gives
\[
f(r(s+h),r(s),h)
=
hA(r(s))r'(s)
+
hB(r(s))
+
o(h).
\]
Substitute this and \eqref{eq:g-small-h} into
\eqref{eq:psi-subdivision}, subtract \(\psi(s)\), divide by \(h\),
and let \(h\to0^+\) to obtain
\begin{equation}\label{eq:psi-ode}
\psi'(s)
=
\frac{1}{1+e^s}\psi(s)
+
\frac{2}{1+e^s}
\left(
A(r(s))r'(s)+B(r(s))
\right).
\end{equation}
Setting
$
\mu(s)=1+e^{-s}
$
gives
$
\mu'(s)=-e^{-s}
$
and
$
\frac{\mu(s)}{1+e^s}=e^{-s}
$.
Multiplying \eqref{eq:psi-ode} by \(\mu(s)\) therefore gives
\[
\frac{d}{ds}\bigl(\mu(s)\psi(s)\bigr)
=
2e^{-s}
\left(
A(r(s))r'(s)+B(r(s))
\right),
\]
and since
$
\psi(0)=f(b,b,0)=0
$,
integration from \(0\) to \(c\) yields
\begin{equation}\label{eq:f-integral}
f(a,b,c)
=
\frac{2}{1+e^{-c}}
\int_0^c
e^{-s}
\left(
A(r(s))r'(s)+B(r(s))
\right)\,ds.
\end{equation}

We now use two degenerate configurations in
\eqref{eq:f-all-boundaries}.  Equation \eqref{eq:f-integral} extends
to them by the assumed \(C^1\) regularity.
First, consider
$
f(t,0,t)=0
$,
where the side data are
$
(a,b,c)=(t,0,t)
$
with
\(R=P\).  Along \(PQ\), we have
$
r(s)=d(X_s,R)=d(X_s,P)=s,
r'(s)=1
$, so
 \eqref{eq:f-integral} gives
\[
0
=
\frac{2}{1+e^{-t}}
\int_0^t e^{-s}
\bigl(A(s)+B(s)\bigr)\,ds,
\]
and so
\[
\int_0^t e^{-s}
\bigl(A(s)+B(s)\bigr)\,ds=0
\]
for every \(t>0\).  Differentiating with respect to \(t\) gives
\begin{equation}\label{eq:alpha-plus-gamma}
A(t)+B(t)=0.
\end{equation}

Next consider
$
f(0,t,t)=0.
$ with
$
(a,b,c)=(0,t,t),
$
so \(R=Q\), and
$
r(s)=d(X_s,Q)=t-s,
r'(s)=-1$.
Equation \eqref{eq:f-integral} gives
\[
0
=
\frac{2}{1+e^{-t}}
\int_0^t e^{-s}
\bigl(
-A(t-s)+B(t-s)
\bigr)\,ds,
\]
whence
\[
\int_0^t e^{-s}
\bigl(
-A(t-s)+B(t-s)
\bigr)\,ds=0.
\]
With the change of variable \(v=t-s\), this becomes
\[
e^{-t}
\int_0^t
e^v\bigl(-A(v)+B(v)\bigr)\,dv=0,
\]
and therefore
\[
\int_0^t
e^v\bigl(-A(v)+B(v)\bigr)\,dv=0.
\]
Differentiating with respect to \(t\) this time gives
\begin{equation}\label{eq:minus-alpha-plus-gamma}
-A(t)+B(t)=0.
\end{equation}

Equations \eqref{eq:alpha-plus-gamma} and
\eqref{eq:minus-alpha-plus-gamma} therefore imply
\[
A(t)=B(t)=0
\]
for every \(t>0\), so
\eqref{eq:f-integral} now gives
$
f(a,b,c)=0
$
for every non-degenerate triangle, as required.
\end{proof}

\subsection{Degree two}

Degree two is the first genuinely nontrivial sector.  The
$\beta\beta$-projection first produces a new projected side density $E_B$ and its
additive completion $\cB_2$, and the kernel of this projection consists of
$\lambda\beta$ terms.  After rewriting these terms as an ordinary one-form on
the body side-length space, additivity forces that one-form to be a multiple
of $d\Area$, producing the second generator
\[
J=\delta_{\!f}\Area\,.
\]
The degree-two space is therefore two-dimensional.  One-sided $C^2$ 
regularity is assumed at the degenerate faces.

\begin{proposition}[The degree-two solution space]
\label{prop:degree-two-class}
The vector space of 
$C^2$-regular cyclic additive scalar invariants 
\[
\begin{aligned}
I_2^{PQR}(a,b,c)={}&
 u(a,b,c)\,\beta_P\beta_Q
 +u(b,c,a)\,\beta_Q\beta_R
 +u(c,a,b)\,\beta_R\beta_P\\
&+\lambda\bigl[
 f(a,b,c)\,\beta_P
 +f(b,c,a)\,\beta_Q
 +f(c,a,b)\,\beta_R
 \bigr]
\end{aligned}
\]
of hyperbolic spin supertriangles in degree two has dimension two.
More explicitly, one generator is the fermionic first variation
\[
J=\delta_{\!f}\Area,
\]
 of the ordinary hyperbolic area of the body
triangle.  Intrinsically, $J=D_{\rm super}-D_{\rm body}$; in common-first
gauge it agrees with the degree-two correction computed in \cite{defect}.
The other generator is characterized by its \(\beta\beta\)-projection
\[
E_B=
\beta_P\beta_Q{c\over\sinh c}
+\beta_Q\beta_R{a\over\sinh a}
+\beta_R\beta_P{b\over\sinh b},
\]
on an arbitrary non-degenerate triangle.  Both generators have even
fermionic degree, so they are fixed by \(\iota\) and descend unchanged to
ordinary supertriangles.

\end{proposition}

\begin{proof}
The following functional equations for additivity $I_2^{PQR}=I_2^{PTR}+I_2^{QRT}$ arise directly by equating coefficients of degree-two monomials:
\begin{equation}\label{tag21}
u(a,b,c)
=
U\,u(m,b,x)
+S\,u(a,m,y),~~~~~~~~~~~~~~~
\end{equation}
\begin{equation}\label{tag22}
u(b,c,a)
=
u(m,y,a)
+U\bigl\{
u(b,x,m)-u(y,a,m)
\bigr\},
\end{equation}
\begin{equation}\label{tag23}
u(c,a,b)
=
u(x,m,b)
-S\bigl\{
u(b,x,m)-u(y,a,m)
\bigr\}\,,
\end{equation}

\[
\begin{aligned}
f(a,b,c)
&=
\ell_R\,u(x,m,b)
-\ell_T\,u(m,b,x)\\
&+
S\Bigl\{
\ell_P\,u(m,b,x)
-\ell_R\,u(b,x,m)
+r_R\,u(y,a,m)
-r_Q\,u(a,m,y)
\Bigr\}\\
&+
g\,f(m,b,x)
+gS\,f(b,x,m)
+dS\,f(a,m,y),
\end{aligned}
\]

\[
\begin{aligned}
f(b,c,a)
&=
-r_R\,u(m,y,a)
+r_T\,u(a,m,y)
\\
&+
U\Bigl\{
\ell_P\,u(m,b,x)
-\ell_R\,u(b,x,m)
+r_R\,u(y,a,m)
-r_Q\,u(a,m,y)
\Bigr\}
\\
&+
gU\,f(b,x,m)
+d\,f(m,y,a)
+dU\,f(a,m,y),\\
\end{aligned}
\]

\[
\begin{aligned}
f(c,a,b)
&=
\ell_T\,u(b,x,m)
-\ell_P\,u(x,m,b)
\\
&+
r_Q\,u(m,y,a)
-r_T\,u(y,a,m)
\\
&+
g\,f(x,m,b)
+d\,f(y,a,m).
\end{aligned}~~~~~~~~~~~~~~~~~~~~~~~~~~~~~~~~~~~~~~~~~~~~
\]

\noindent The first three additivity
equations are independent of $f$, and we
%\begin{align}  u(a,b,c) &=Uu(m,b,x)+Su(a,m,y),                                      \label{eq:u-one-proof}\\
% u(b,c,a) &=u(m,y,a)+U\{u(b,x,m)-u(y,a,m)\},                         \label{eq:u-two-proof}\\ 
%u(c,a,b)&=u(x,m,b)-S\{u(b,x,m)-u(y,a,m)\},                         \label{eq:u-three-proof}\end{align}
and begin by classifying the $\beta\beta$ projection of $I_2^{PQR}$ which they determine.
%where for convenience we recall from Section \ref{sec:subdivision-calculus} that $a=d(Q,R),~b=d(P,R),~c=d(P,Q)=x+y, ~ U={\sinh x\over\sinh c},~ S={\sinh y\over\sinh c}.$

\medskip
\noindent
\textit{Step 1: The first equation
 gives a one-variable side density.}
Define
$$W(P,Q;R)=\sinh c\,u(a,b,c).$$ Multiplying \eqref{tag21} by $\sinh c$ gives
\[
W(P,Q;R)=W(P,T;R)+W(T,Q;R),
\]
so $W$ is additive when the geodesic segment $PQ$ is subdivided.
Define the degenerate trace
\[
\rho(r)=u(r,r,0).
\]
If $X_h$ is distance $h$ from $X$ along a geodesic, regularity gives
\[
W(X,X_h;R)
 =\sinh h\,u(d(X_h,R),d(X,R),h)
 =h\,\rho(d(X,R))+o(h).
\]
Subdivide $PQ$ into $N$ pieces, use additivity, and let the mesh tend to
zero.  The resulting Riemann sum gives
\begin{equation}
{
 u(a,b,c)
 ={1\over\sinh c}\int_{PQ}\rho(d(X,R))\,ds.}
\label{eq:rho-density-proof}
\end{equation}
No further geometric information enters this formula: the infinitesimal
density is exactly the regular degenerate trace $u(r,r,0)$.

\medskip
\noindent
\textit{Step 2: The second and third equations force the density to be constant.}
Take a collinear degeneration with the points ordered
$
R\,\--\;P\--\,\;T\--\;Q, 
$
and write
$$
d(P,T)=x=f c, ~~
d(T,Q)=y=(1-f)c\,,~~{\rm for}
~0<f<1,$$
so that
$$
a=RQ=b+c,~~
{\rm and}~~
m=RT=b+f c.$$
Now with $\rho$ as above, let
\[
I(t)=\int_0^t\rho(s)\,ds.
\]
Formula \eqref{eq:rho-density-proof} directly gives the following identities
on this degenerate
configuration:
\[
 u(b,c,a)={I(b)+I(c)\over\sinh(b+c)},
\]
\[
 u(m,y,a)={I(b+f c)+I((1-f)c)\over\sinh(b+c)},
\]
\[
 u(b,x,m)={I(b)+I(f c)\over\sinh(b+f c)},
\]
and
\[
 u(y,a,m)={I(b+c)-I((1-f)c)\over\sinh(b+f c)}.
\]
For instance for the first formula, the cyclically permuted coefficient
\(u(b,c,a)\) is obtained from \eqref{eq:rho-density-proof} by integrating
along the side \(QR\), whose opposite vertex is \(P\).  In the collinear
configuration
$
R\;-\;P\;-\;T\;-\;Q
$,
we have
$
RP=b,~ PQ=c,~RQ=a=b+c,$
so
\[
u(b,c,a)
=
\frac{1}{\sinh a}
\int_{QR}\rho\bigl(d(X,P)\bigr)\,ds.
\]
Splitting the integral at \(P\), the distance \(d(X,P)\) decreases from
\(b\) to \(0\) along \(RP\) and then increases from \(0\) to \(c\) along
\(PQ\).  Therefore
\[
\int_{QR}\rho\bigl(d(X,P)\bigr)\,ds
=
\int_0^b\rho(r)\,dr+\int_0^c\rho(r)\,dr
=
I(b)+I(c).
\]
Since \(a=b+c\), this gives the first formula, and the others are analogous.

Substitution in \eqref{tag22} therefore directly gives the exact identity
\begin{align}
0={}&
{I(b)+I(c)-I(b+f c)-I((1-f)c)\over\sinh(b+c)}
\notag\\
&-{\sinh(f c)\over\sinh c\,\sinh(b+f c)}
\biggl\{I(b)+I(f c)-I(b+c)+I((1-f)c)\biggr\}\,,
\label{eq:exact-collinear-proof}
\end{align}
and setting $
\rho_0=\rho(0)$
and
$\rho_b=\rho(b)$ yields
\[
\begin{aligned}
I(c)&=\rho_0c+{1\over2}\rho'(0)c^2+O(c^3),\\
I(f c)&=\rho_0f c+{1\over2}\rho'(0)f^2c^2+O(c^3),\\
I((1-f)c)&=\rho_0(1-f)c
 +{1\over2}\rho'(0)(1-f)^2c^2+O(c^3),\\
I(b+f c)&=I(b)+\rho_bf c
 +{1\over2}\rho'(b)f^2c^2+O(c^3),\\
I(b+c)&=I(b)+\rho_b\,c+{1\over2}\rho'(b)c^2+O(c^3).
\end{aligned}
\]
Furthermore,
\[
{\sinh(f c)\over\sinh c}=f+O(c^2),
\]
\[
{1\over\sinh(b+c)}
 ={1\over\sinh b}-{\cosh b\over\sinh^2b}c+O(c^2),
\]
and
\[
{1\over\sinh(b+f c)}
 ={1\over\sinh b}
 -f{\cosh b\over\sinh^2b}c+O(c^2),
\]
so the coefficient of $c$ in \eqref{eq:exact-collinear-proof} cancels
identically.  The coefficient of $c^2$ is
\[
{f(1-f)\over2\sinh^2b}
\Bigl[
2\cosh b\,(\rho_b-\rho_0)
+\sinh b\,\rho'(b)
+(2f+1)\sinh b\,\rho'(0)
\Bigr].
\]
Since \eqref{eq:exact-collinear-proof} holds for every
$0<f<1$, its coefficient of $f$ inside the square bracket gives
$
\rho'(0)=0$.
Writing $q(b)=\rho(b)-\rho(0)$ in the remaining equation
$\sinh b\,\rho'(b)
+2\cosh b\,[\rho(b)-\rho(0)]=0$, we obtain
$
q'(b)+2\coth b\,q(b)=0,
$ and hence $q(b)={C\over\sinh^2b}$.
Regularity at $b=0$ forces $C=0$, so $\rho$ is some constant
$\rho\equiv K$.
Thus,  \eqref{eq:rho-density-proof} becomes
\[
{u(a,b,c)=K{c\over\sinh c},}
\]
and the projected $\beta\beta$ solution space is consequently one-dimensional.

\bigskip

For $u_0(a,b,c)=\frac{c}{\sinh c}$, the three projected equations 
\[
U{x\over\sinh x}+S{y\over\sinh y}
={x+y\over\sinh c}={c\over\sinh c},
\]
are verified
without approximation,
while in the second and third equations the two internal terms
$\frac{m}{\sinh m}$ cancel.  As was claimed, the corresponding projected invariant is therefore
\[
E_B=
\beta_P\beta_Q{c\over\sinh c}
+\beta_Q\beta_R{a\over\sinh a}
+\beta_R\beta_P{b\over\sinh b}.
\]

\medskip
\noindent
\textit{Step 3: The projected solution $E_B$ has a full additive completion.}

Define the full subdivision defect of \(E_B\) across the cut at \(T\) by
\begin{equation}\label{eq:defect-B}
\begin{aligned}
\mathcal S(P,T,Q;R)
={}&E_B^{PQR}
-\mathsf T_{PTR\to PQR}E_B^{PTR}
-\mathsf T_{TQR\to PQR}E_B^{TQR},
\end{aligned}
\end{equation}
where $\mathsf T$ is the flat transport operator in the previous section.
The projected \(\beta\beta\)-equations have already been verified exactly,
so every \(\beta\beta\)-term cancels in \eqref{eq:defect-B}, whence
\[
\mathcal S(P,T,Q;R)
=
\lambda\bigl(
A_P\beta_P+A_Q\beta_Q+A_R\beta_R
\bigr)
\]
for ordinary scalar coefficient functions \(A_P,A_Q,A_R\).  These
coefficients have the required one-sided regularity because they are obtained
from \(\frac{t}{\sinh t}\) and the explicit common-first transport formulas by
algebraic substitution.  Moreover,
$
\mathcal S(P,P,Q;R)=0=
\mathcal S(P,Q,Q;R)
$
in the one-sided sense because at either endpoint one daughter collapses and the
other daughter is the parent.

We shall use an elementary two-cut identity.  If \(P,T,U,Q\) occur in this
order on \(PQ\), transport every term to the common-first coordinates of
\(PQR\), so that
\begin{equation}\label{eq:two-cut-B}
\mathcal S(P,U,Q;R)+\mathcal S(P,T,U;R)
=
\mathcal S(P,T,Q;R)+\mathcal S(T,U,Q;R).
\end{equation}
Indeed, expanding either side by \eqref{eq:defect-B} and using
flatness gives the same four transported terms,
namely the parent term minus the three final daughter terms.

Now parametrize \(PQ\) by arclength,
\[
X_s\in PQ,
\qquad X_0=P,
\qquad X_c=Q,
\qquad d(P,X_s)=s,
\]
and for \(0\le r<t\le c\), define
\[
\Gamma(r,t)
=
\left.
\frac{\partial}{\partial h}
\right|_{h=0^+}
\mathcal S(X_r,X_{r+h},X_t;R).
\]
The derivative exists by the assumed boundary regularity and the endpoint
vanishing just noted.  It is of type \(\lambda\beta\) in the common-first
coordinates of \(X_rX_tR\).

For \(0\le s<t\le c\), set
\begin{equation}\label{eq:C-B}
\mathcal C(X_s,X_t,R)
=
\int_s^t
\mathsf T_{X_rX_tR\to X_sX_tR}\,\Gamma(r,t)\,dr.
\end{equation}
Every integrand is therefore expressed in the fixed common-first coordinates
of \(X_sX_tR\) before integration.
We claim that the subdivision defect of \(\mathcal C\) is exactly
\(\mathcal S\).  Fix \(s<u<t\), transport all terms to the common-first
coordinates of \(X_sX_tR\), and apply \eqref{eq:two-cut-B} to
\[
X_r,\quad X_{r+h},\quad X_u,\quad X_t,
\qquad s\le r<r+h<u<t.
\]
After rearranging, dividing by \(h\), and letting \(h\downarrow0\), one gets
\[
\frac{d}{dr}\mathcal S(X_r,X_u,X_t;R)
=
\Gamma(r,u)-\Gamma(r,t).
\]
Hence, using \eqref{eq:C-B},
\[
\begin{aligned}
&\mathcal C(X_s,X_t,R)
-\mathcal C(X_s,X_u,R)
-\mathcal C(X_u,X_t,R)\\
&\qquad=
\int_s^u\bigl(\Gamma(r,t)-\Gamma(r,u)\bigr)\,dr\\
&\qquad=
-\int_s^u
\frac{d}{dr}\mathcal S(X_r,X_u,X_t;R)\,dr\\
&\qquad=
\mathcal S(X_s,X_u,X_t;R),
\end{aligned}
\]
since \(\mathcal S(X_u,X_u,X_t;R)=0\).  Thus, we have
\begin{equation}\label{eq:C-defect-B}
\mathcal C(X_s,X_t,R)
-\mathcal C(X_s,X_u,R)
-\mathcal C(X_u,X_t,R)
=
\mathcal S(X_s,X_u,X_t;R).
\end{equation}

Define
\[
B_2^{PQR}=E_B^{PQR}-\mathcal C(P,Q,R).
\]
For a cut \(T\in PQ\), equations \eqref{eq:defect-B} and
\eqref{eq:C-defect-B} give zero subdivision defect for \(B_2\).  Hence
\(B_2\) is additive for cuts of \(PQ\).
Meanwhile, the whole construction is equivariant under the cyclic relabelling
\[
(P,Q,R)\longmapsto(Q,R,P),
\]
and \(E_B\) itself is cyclic.  Consequently
\[
\mathcal C(P,Q,R)
=
\lambda\bigl(
 v(a,b,c)\beta_P
+v(b,c,a)\beta_Q
+v(c,a,b)\beta_R
\bigr)
\]
for a regular scalar function \(v\), so \(B_2\) is cyclic.  Applying the
already proved additivity for the first side after cyclic relabelling gives
additivity for cuts of \(QR\) and \(RP\) as well.  Finally,
\(\mathcal C\) has type \(\lambda\beta\), and therefore
$
(B_2)_{\beta\beta}=E_B$,
so \(E_B\) has a regular cyclic additive completion.  The difference of any two such completions has zero
\(\beta\beta\)-projection and is handled in the next step.

\medskip
\noindent
\textit{Step 4: The kernel of the $\beta\beta$ projection is one-dimensional.}
Consider 
\[
I_2^{PQR}
=\lambda\biggl[f(a,b,c)\beta_P+f(b,c,a)\beta_Q+f(c,a,b)\beta_R\biggr],
\]
and let
\[
\widehat a=a+\Delta a,\qquad
\widehat b=b+\Delta b,\qquad
\widehat c=c+\Delta c
\]
be the three full even side coordinates in the common-first normalization, 
where
$$
\Delta a=-\lambda\, (\beta_Q+\beta _R)\,{\rm tanh}\frac{a}{2},~~~
\Delta b=-\lambda\, (\beta_R+\beta _P)\,{\rm tanh}\frac{b}{2},~~~
\Delta c=-\lambda\, (\beta_P+\beta _Q)\,{\rm tanh}\frac{c}{2}.
$$
Put
\[
X=\lambda\beta_P,
\qquad Y=\lambda\beta_Q,
\qquad Z=\lambda\beta_R,
\]
so that
\[
\begin{pmatrix}
\Delta a\\ \Delta b\\ \Delta c
\end{pmatrix}
=
-\begin{pmatrix}
0&\tanh(a/2)&\tanh(a/2)\\
\tanh(b/2)&0&\tanh(b/2)\\
\tanh(c/2)&\tanh(c/2)&0
\end{pmatrix}
\begin{pmatrix}X\\Y\\Z\end{pmatrix}.
\]
The determinant of this change of variables is
$
-2\tanh(a/2)\tanh(b/2)\tanh(c/2)$,
which is nonzero on the non-degenerate stratum.  Hence there is a unique
ordinary one-form
\[
\omega=\omega_a\,da+\omega_b\,db+\omega_c\,dc
\]
such that
\[
I_2^{PQR}
=\omega_a\Delta a+\omega_b\Delta b+\omega_c\Delta c.
\]

After setting \(u=0\), the last three displayed degree-two subdivision
equations which began this proof are homogeneous equations for \(f\).  Substituting the invertible
relations above, those three equations are exactly the additivity equations
for \(\omega\), with the daughter one-forms pulled back by the corresponding
body side-length maps.  We now prove directly that \(\omega\) is the
differential of an ordinary additive scalar function. 
To this end, let
\[
{\mathcal M}
=
\left\{
(a,b,c)\in{\mathbb R}_{>0}^3:
a<b+c,\quad b<c+a,\quad c<a+b
\right\}
\]
be the ordinary side-length domain of non-degenerate hyperbolic
triangles, and notice that \({\mathcal M}\) is convex as the
intersection of open linear half-spaces.

We first show that the one-form $\omega$
is closed.
To see this, fix a non-degenerate triangle \(PQR\), and let
\(A,B,C\) denote its body angles at \(P,Q,R\), respectively, as in Figure \ref{fig:body}.
There are three elementary infinitesimal side cuts, as follows.
If \(Q\) is moved a distance \(s\) toward \(P\) along \(QP\), then
to first order, 
$dc=-ds,~ db=0,~ da=-\cos B\,ds$, so
$X_Q=(-\cos B,0,-1)$
is the corresponding tangent vector in \((a,b,c)\)-coordinates.
Similarly, moving \(R\) toward \(Q\) along \(RQ\) gives
$
X_R=(-1,-\cos C,0),
$
and moving \(P\) toward \(R\) along \(PR\) gives
$
X_P=(0,-1,-\cos A)$.
These three vectors are linearly independent.  Indeed,
\[
\det
\begin{pmatrix}
-\cos B&-1&0\\
0&-\cos C&-1\\
-1&0&-\cos A
\end{pmatrix}
=
-\bigl(1+\cos A\cos B\cos C\bigr),
\]
and this is nonzero for a non-degenerate triangle, since
$|\cos A\cos B\cos C|<1$.
It is therefore enough to prove that
$
d\omega(X_i,X_j)=0
$
for each pair among \(X_P,X_Q,X_R\).

For this, consider two sufficiently small side cuts, of sizes \(s\) and \(t\),
corresponding to two of these three infinitesimal motions.  Perform
the \(s\)-cut first and then the \(t\)-cut.  Alternatively, perform
the \(t\)-cut first and then the \(s\)-cut.  In both cases, refine the
resulting decomposition to the common subdivision obtained by making
both cuts.

The additivity equation for \(\omega\) says that its value on the
original triangle is the sum of its values on the daughter triangles,
with the daughter coefficients rewritten in the coordinates of the
original triangle.  After the two cuts, the two orders therefore give
the same sum because every final daughter triangle occurs once in each
decomposition, and the common-first substitutions compose
associatively by flatness.

The two orders define the boundary of a small two-parameter loop in
\({\mathcal M}\), whose tangent directions at the initial triangle are
\(X_i\) and \(X_j\).  Because the two final decompositions are literally the
same common refinement, flatness of common-first transport identifies every
final daughter term in the two calculations.  The contributions along the
interior subdivision faces therefore occur with opposite orientations and
cancel exactly; the replacement of the resulting curved four-sided loop by
its tangent parallelogram changes its integral only by \(o(st)\).  Thus,
\[
\oint\omega=o(st).
\]
On the other hand, the ordinary first-order expansion of a \(C^1\)
one-form around the tangent parallelogram gives
\[
\oint\omega
=
st\,d\omega(X_i,X_j)+o(st).
\]
It follows that $d\omega(X_i,X_j)=0$, and
since \(X_P,X_Q,X_R\) span the tangent space of \({\mathcal M}\),
we conclude that $\omega$ is indeed closed.
Since \({\mathcal M}\) is convex and \(d\omega=0\), the Poincaré lemma gives a scalar function
$
\Phi:{\mathcal M}\to\mathbb R
$
such that $\omega=d\Phi$.

We also need the boundary pullback of \(\omega\).  In the third homogeneous
\(f\)-equation, let \(T\to P\), so that  \(x\to0\), \(y\to c\),
\(m\to b\), and the equation reduces to
$
f(0,b,b)=0$.
On the face \(c=0\), \(a=b=t\), we have \(\Delta c=0\), and the
coefficient of \(\lambda\beta_R\) in
\(I_2=\omega_a\Delta a+\omega_b\Delta b+\omega_c\Delta c\) is
$
-\tanh(t/2)\bigl(\omega_a+\omega_b\bigr)$.
Since this coefficient is \(f(0,t,t)=0\), it follows that the pullback of
\(\omega\) to this face is
\[
\iota_c^*\omega
=\bigl(\omega_a+\omega_b\bigr)\,dt=0.
\]
The two cyclically related faces are identical, and therefore the pullback of
\(\omega=d\Phi\) to every degenerate face vanishes, so \(\Phi\) is constant
on each such face.  Moreover \(\omega\) is cyclic, and so
\[
d\bigl(
\Phi(b,c,a)-\Phi(a,b,c)
\bigr)=0.
\]
Since \({\mathcal M}\) is connected, the difference is a constant.
Applying the cyclic permutation three times shows that three times
this constant is zero, so the constant itself is zero.  Thus, 
\(\Phi\) is moreover cyclic, and the three constants on the three degenerate
faces are consequently equal.  Subtracting this common constant
from \(\Phi\), we may and shall
assume
\begin{equation}\label{eq:Phi-boundary-zero}
\Phi(t,t,0)=\Phi(t,0,t)=\Phi(0,t,t)=0.
\end{equation}

It remains only  to prove that \(\Phi\) is additive.
For a cut \(T\in PQ\), define its subdivision defect by
\[
D_\Phi(P,T,Q;R)
=
\Phi(P,Q,R)
-
\Phi(P,T,R)
-
\Phi(T,Q,R).
\]
All three terms are regarded as functions of the ordinary body
configuration.
Taking the ordinary differential gives
\[
\begin{aligned}
dD_\Phi
&=
d\Phi(P,Q,R)
-
d\Phi(P,T,R)
-
d\Phi(T,Q,R)\\
&=
\omega(P,Q,R)
-
\omega(P,T,R)
-
\omega(T,Q,R).
\end{aligned}
\]
The last expression is zero, because the homogeneous
\(f\)-equations just identified are precisely the subdivision equations
satisfied by \(\omega\).  It follows that
$dD_\Phi=0$, and since
the one-cut configuration space is connected, \(D_\Phi\) must be
constant.

Now let \(T\to P\).  The triangle \(PTR\) collapses, while
\(TQR\to PQR\).  By \eqref{eq:Phi-boundary-zero}, $\Phi(P,T,R)\longrightarrow0$,
and hence $
D_\Phi(P,P,Q;R)
=
\Phi(P,Q,R)-0-\Phi(P,Q,R)=0$.
The constant is therefore zero, and so
\[
\Phi(P,Q,R)
=
\Phi(P,T,R)+\Phi(T,Q,R)
\]
for every side cut \(T\in PQ\).  By cyclicity, the same holds for
cuts of the other two sides.

Thus, \(\Phi\) is an ordinary regular cyclic additive scalar invariant.  By
Proposition~\ref{prop:degree0-main}, there is a constant \(k\) such that
$
\Phi=k\,\Area$,
so
\[
\omega=d\Phi=k\,d\Area.
\]
Evaluating this one-form on the actual degree-two common-first side
variation
$
(\Delta a,\Delta b,\Delta c)
$
gives
\[
\begin{aligned}
I_2
&=
\omega_a\Delta a+\omega_b\Delta b+\omega_c\Delta c\\
&=
k\,d\Area(\Delta a,\Delta b,\Delta c).
\end{aligned}
\]
Define
\[
J
=
\delta_{\!f}\Area
=
d\Area(\Delta a,\Delta b,\Delta c),
\]
so 
$
I_2=kJ$, and
the kernel of the \(\beta\beta\)-projection is indeed one-dimensional
and spanned by \(J\).

%Finally, let \(I_2\) be any degree-two additive invariant.  Its \(\beta\beta\)-projection is \(B E_B\) for a unique scalar \(B\). Subtracting \(B\cB_2\) leaves an invariant in the kernel of the projection, hence a unique multiple \(A J\).  Therefore \[ I_2=A J+B\cB_2. \] Since \(J\) has zero \(\beta\beta\)-projection whereas \(\cB_2\) has projection \(E_B\neq0\), these two generators are linearly independent.

\medskip
\noindent
\textit{Step 5: Relation with angle defect and with \cite{defect}.}
The degree-two classification is complete at this point.  
We record separately the
angle-defect interpretation of \(J\), taking care to distinguish the body
area from the super area used in \cite{defect}.

Let
$
\Omega(a,b,c)=\Area(P_\#,Q_\#,R_\#)
$
be the ordinary hyperbolic area, equivalently the ordinary angle defect, of
the body triangle.  In common-first coordinates the full even side
coordinates are
$
\widehat a=a+\Delta a,
~
\widehat b=b+\Delta b,
~
\widehat c=c+\Delta c$.
Since the hyperbolic law of cosines has the same analytic form for the full
even side coordinates of a supertriangle, its super angle defect is given by
$
D_{\rm super}
=
\Omega(\widehat a,\widehat b,\widehat c)$.
If
$
D_{\rm body}=\Omega(a,b,c)$,
then every product of two of \(\Delta a,\Delta b,\Delta c\) vanishes,
because each contains the common factor \(\lambda\).  Hence the nilpotent
Taylor expansion is exactly linear and gives
\[
\begin{aligned}
D_{\rm super}-D_{\rm body}
&=
\Omega(a+\Delta a,b+\Delta b,c+\Delta c)-\Omega(a,b,c)\\
&=
d\Omega(\Delta a,\Delta b,\Delta c)\\
&=
\delta_{\!f}\Area
=J.
\end{aligned}
\]
Thus, the invariant isolated above is precisely the first fermionic
transgression from the body angle defect to the super angle defect.

It remains to compare this with the calculation of \cite{defect}.
The two-form used there to define a super area is not invariant under
the odd part of $\osp$, so the resulting area functional does not define
an intrinsic scalar of arbitrary supertriangles.  Nevertheless, for
triangles already expressed in common-first gauge, the bosonization
calculation used there remains valid.  With the sign convention of
Theorem~4.3 and Lemma~4.4 of \cite{defect}, the bosonization map is
\[
X=(x_1,x_2,x\mid\varphi,\psi)
\longmapsto
\widehat X
=
(1+\varphi\psi)(x_1,x_2,x\mid0,0),
\]
and
$
\widehat x_1\widehat x_2-\widehat x^2
=
(1+2\varphi\psi)(1-2\varphi\psi)=1$.
For a common-first vertex
$
P=((1-\lambda\beta_P)y_P\mid\lambda,\beta_P)
$,
we therefore have
\[
\widehat P
=
(1+\lambda\beta_P)(1-\lambda\beta_P)y_P
=
y_P,
\]
and similarly
$
\widehat Q=y_Q$ and  $\widehat R=y_R$.

Moreover, bosonization sends each supergeodesic edge to the
corresponding ordinary geodesic edge.  Indeed, the even part of every
point of the supergeodesic \(PQ\) lies in
\(\operatorname{span}(y_P,y_Q)\), and bosonization only multiplies
that even part by an even scalar.  The bosonized point has ordinary
Lorentz norm one, and hence lies on
$
\mathbb H^2\cap\operatorname{span}(y_P,y_Q)$, that is, on
the ordinary geodesic joining \(y_P\) to \(y_Q\).
The same argument applies to the other two sides.

Thus, for a triangle already in common-first gauge, the area functional
used in \cite{defect}, which we denote here by \(A_{\rm old}\), satisfies
\[
A_{\rm old}(P,Q,R)
=
\Area(y_P,y_Q,y_R)
=
\Omega(a,b,c)
=
D_{\rm body}.
\]
Consequently, within common-first gauge,
\[
\begin{aligned}
D_{\rm super}-A_{\rm old}
&=
D_{\rm super}-D_{\rm body}\\
&=
d\Omega(\Delta a,\Delta b,\Delta c)\\
&=
\delta_{\!f}\Area\\
&=
J.
\end{aligned}
\]
This is the degree-two formula computed in \cite{defect},
which does actually disprove the Angle Defect Theorem but only for gauged
supertriangles.  The intrinsic identity established here is
\[
{J=D_{\rm super}-D_{\rm body}},
\]
which is independent of any choice of super-area form.
The exact cancellation in common-first gauge explains why the
degree-two discrepancy computed in \cite{defect} survives as the
intrinsic additive invariant \(J\), even though the old area form is
not globally \(\operatorname{OSp}(1|2)\)-invariant. \end{proof}

\subsection{Degree three}

We first classify the cubic $\beta_P\beta_Q\beta_R$ projection, then prove
that its kernel is zero, and finally construct the unique full additive
completion $\cK_3$.  Only \(C^1\) one-sided regularity at the degenerate body faces is required,
including at the common corner \(a=b=c=0\).

\begin{proposition}[The degree-three solution space]
\label{prop:degree-three-class}
The vector space of \(C^1\)-regular cyclic additive degree-three scalar
invariants 
\[
\begin{aligned}
I_3^{PQR}={}&
 v(a,b,c)\,\beta_P\beta_Q\beta_R\\
&+\lambda\Bigl\{
 u(a,b,c)\,\beta_P\beta_Q
 +u(b,c,a)\,\beta_Q\beta_R
 +u(c,a,b)\,\beta_R\beta_P
 \Bigr\}
\end{aligned}
\]
of hyperbolic spin supertriangles is one-dimensional.  A generator is
\[
\cK_3^{PQR}
=
K(a,b,c)\,\beta_P\beta_Q\beta_R
+
\lambda\Bigl\{
U_0(a,b,c)\beta_P\beta_Q
+U_0(b,c,a)\beta_Q\beta_R
+U_0(c,a,b)\beta_R\beta_P
\Bigr\},
\]
where
\[
K(a,b,c)
=
\frac1{\sinh c}
\int_{PQ}\frac{2\,ds}{1+\cosh d(X,R)}.
\]
For a non-degenerate body triangle this may equivalently be written
\[
K(a,b,c)
=
\frac{2\,\Area(P,Q,R)}{\sinh h_c\,\sinh c},
\]
where \(h_c\) is the altitude from \(R_\#\) to the edge \(P_\#Q_\#\) taken with a plus/minus
sign if $R_\#$ lies to the left/right of the oriented segment from $P_\#$ to $Q_\#$.
The integral formula gives
the regular extension to the degenerate faces.  The lower coefficient
\(U_0\) is uniquely determined by additivity.  Moreover,
$
\iota^*\cK_3=-\cK_3$,
so this one-dimensional odd space exists only on the spin refinement and
does not descend to ordinary supertriangles.
\end{proposition}

\begin{proof}
%For a cut \(T\in PQ\), retain the notation of Section \ref{sec:subdivision-calculus}.
%\[ a=d(Q,R),\qquad b=d(P,R),\qquad c=x+y=d(P,Q),\] \[ x=d(P,T),\qquad y=d(T,Q),\qquad m=d(T,R), \] and \[U=\frac{\sinh x}{\sinh c}, \qquad S=\frac{\sinh y}{\sinh c}. \]
Equating the four degree-three coefficients in
\(I_3^{PQR}=I_3^{PTR}+I_3^{TQR}\) using the full common-first transport
gives
\begin{align}
 v(a,b,c)
&=U\,v(m,b,x)+S\,v(m,y,a),
\label{eq:d3-v}\\
 u(a,b,c)
&=U\{g\,u(m,b,x)+\ell_Rv(m,b,x)\}
 +S\{d\,u(a,m,y)+r_Rv(m,y,a)\},
\label{eq:d3-u1}\\
 u(b,c,a)
&=gU\,u(b,x,m)+U\ell_Pv(m,b,x)\notag\\
&\qquad
 +d\{u(m,y,a)-Uu(y,a,m)\}
 +(r_T-Ur_Q)v(m,y,a),
\label{eq:d3-u2}\\
 u(c,a,b)
&=g\{u(x,m,b)-Su(b,x,m)\}
 +(\ell_T-S\ell_P)v(m,b,x)\notag\\
&\qquad
 +dS\,u(y,a,m)+Sr_Qv(m,y,a).
\label{eq:d3-u3}
\end{align}

\medskip
\noindent
\textit{Step 1: The cubic projection is one-dimensional.}
Let
\[
\Delta(P,Q,R)=\det(P_\#,Q_\#,R_\#)
\]
be the ordinary determinant in the hyperboloid model.  Since
$
T_\#=S P_\#+U Q_\#$,
multilinearity gives
\[
\Delta(P,T,R)=U\Delta(P,Q,R),
\qquad
\Delta(T,Q,R)=S\Delta(P,Q,R).
\]
Multiplying \eqref{eq:d3-v} by \(\Delta(P,Q,R)\) therefore shows that
\[
A(P,Q,R)=\Delta(P,Q,R)v(a,b,c)
\]
is an ordinary regular cyclic additive scalar invariant of the body
triangle.  It vanishes on a collapsed side because \(\Delta=0\) there.
The degree-zero argument, for which continuity already suffices, therefore
gives
\[
A(P,Q,R)=\kappa\,\Area(P,Q,R)
\]
for some constant \(\kappa\).  Hence the possible coefficients \(v\) form a
one-dimensional space.

In the side-normal body frame, we have
$
P_\#=(1,1,0),
~
Q_\#=(e^c,e^{-c},0)$,
and with the signed altitude convention above,
\[
\Delta(P,Q,R)=-2\sinh h_c\,\sinh c.
\]
After absorbing the fixed sign and scalar factor into \(\kappa\), we may take
the projected generator to be
\[
K(a,b,c)=\frac{2\Area(P,Q,R)}{\sinh h_c\,\sinh c}.
\]
The equivalent side-density formula is useful because it makes projected
additivity immediate.  

If \(F\) is the foot of the altitude from \(R\) to the geodesic \(PQ\),
and \(s\) is signed arclength along \(PQ\) measured from \(F\), then
for the point \(X_s\in PQ\), we find
$
\cosh d(X_s,R)=\cosh h_c\,\cosh s
$
by the hyperbolic right-triangle cosine law, so
\[
\int\frac{2\,ds}{1+\cosh h_c\cosh s}
=
\frac{4}{\sinh h_c}
\arctan\!\left(
\tanh\frac{h_c}{2}\tanh\frac{s}{2}
\right).
\]
Taking the difference at the two endpoints, or equivalently using the
ordinary right-triangle area formula, gives
\[
\int_{PQ}\frac{2\,ds}{1+\cosh d(X,R)}
=
\frac{2\Area(P,Q,R)}{\sinh h_c},
\]
so that
\[
K(a,b,c)
=
\frac1{\sinh c}
\int_{PQ}\frac{2\,ds}{1+\cosh d(X,R)}.
\]
Splitting the integral at \(T\) now gives directly
\[
K(a,b,c)=U K(m,b,x)+S K(m,y,a),
\]
which is precisely \eqref{eq:d3-v}.

\medskip
\noindent
\textit{Step 2: The kernel of the cubic projection is zero.}
If \(v=0\), then \eqref{eq:d3-u1}--\eqref{eq:d3-u3} reduce to
\[\begin{aligned}
 u(a,b,c)&=Ug\,u(m,b,x)+Sd\,u(a,m,y),\\
 u(b,c,a)&=gU\,u(b,x,m)
 +d\{u(m,y,a)-Uu(y,a,m)\},\\
 u(c,a,b)&=g\{u(x,m,b)-Su(b,x,m)\}
 +dS\,u(y,a,m).
\end{aligned}\]
The common-first factors are
\[
g=\frac{e^{y/2}\cosh(x/2)}{\cosh(c/2)},
\qquad
d=\frac{e^{-x/2}\cosh(y/2)}{\cosh(c/2)}.
\]
Set
\[
N(t)=\frac{e^{t/2}}{\sinh t\,\cosh(t/2)},
\qquad
F(a,b,c)=\frac{u(a,b,c)}{N(c)}.
\]
Since
\[
Ug=\frac{N(c)}{N(x)},
\qquad
Sd=e^{-x}\frac{N(c)}{N(y)},
\]
the first homogeneous equation becomes
\begin{equation}
F(a,b,c)=F(m,b,x)+e^{-x}F(a,m,y)
\label{eq:d3-weighted}
\end{equation}
in analogy to the degree-one case.

Define the collapsed-side trace
\[
\sigma(t)=u(t,t,0).
\]
If \(X_s\in PQ\) is the arclength parametrization from \(P\), put
\(r(s)=d(X_s,R)\).  Since
\[
N(h)=\frac1h+O(1),
\qquad h\downarrow0,
\]
a short segment of length \(h\) contributes
\[
F\bigl(d(X_{s+h},R),d(X_s,R),h\bigr)
=h\,\sigma(r(s))+o(h).
\]
Repeated application of \eqref{eq:d3-weighted} to a partition of \(PQ\)
and passage to the Riemann sum therefore gives
\begin{equation}
F(a,b,c)=\int_0^c e^{-s}\sigma(r(s))\,ds,
\qquad
u(a,b,c)=N(c)\int_0^c e^{-s}\sigma(r(s))\,ds,
\label{eq:d3-density}
\end{equation}
and it remains only to show that \(\sigma=0\).

With $
L(t)=u(0,t,t)$, taking
 \(R=Q\) in \eqref{eq:d3-density} gives
\begin{equation}
L(t)=N(t)\,e^{-t}\int_0^t e^s\sigma(s)\,ds.
\label{eq:d3-L-density}
\end{equation}
Now take the collinear triangle
\[
(a,b,c)=(c,0,c),
\]
so that \(R=P\).  For a cut with \(x+y=c\), one has \(m=x\).
The third homogeneous equation gives the exact identity
\begin{equation}
\sigma(c)
=
g\{\sigma(x)-S L(x)\}
+dS\,u(y,c,x).
\label{eq:d3-collinear}
\end{equation}

First, let \(x\downarrow0\), so \(y=c-x\), and put
\(\tau=\tanh(c/2)\).  Formulae
\eqref{eq:d3-density}--\eqref{eq:d3-L-density} give
\[
L(x)=\sigma(0)+\frac12\sigma'(0)x+o(x),
\]
\[
u(c-x,c,x)=\sigma(c)-\frac12\sigma'(c)x+o(x).
\]
Moreover,
$
S=1-x\coth c+o(x)$, and
\[
g=(1+\tau)\left(1-\frac x2\right)+o(x),
\qquad
d=1-\frac{1+\tau}{2}x+o(x).
\]
Comparing the coefficient of \(x\) in \eqref{eq:d3-collinear} yields
\begin{equation}
\sigma'(c)
=(1+\tau)\sigma'(0)
+2(1+\tau)\coth c\,\sigma(0)
-(1+\tau+2\coth c)\sigma(c).
\label{eq:d3-sigma-first}
\end{equation}
Since
\[
\tau=\frac c2+o(c),\qquad
\coth c=\frac1c+o(1),
\qquad
\sigma(c)=\sigma(0)+c\sigma'(0)+o(c),
\]
the right-hand side of \((51)\) tends to \(-\sigma'(0)\), anc
the left-hand side to \(\sigma'(0)\), so
$
\sigma'(0)=0$.

We next use the second homogeneous equation on the same collinear triangle.
As \(x\downarrow0\), the density formula gives
\[
\begin{aligned}
u(x,c-x,c)
={}&L(c)\\
&+x\Bigl[
L(c)+N(c)e^{-c}\sigma(0)-N(c)\sigma(c)
\Bigr]+o(x),
\end{aligned}
\]
while
\[
U=\frac{x}{\sinh c}+o(x),
\qquad
L(x)=\sigma(0)+o(1).
\]
Substitution in the second homogeneous equation and comparison of the
coefficient of \(x\) gives
\begin{equation}
L(c)
=
\frac{(3e^c+1)\sigma(c)-2(e^c+1)\sigma(0)}{\sinh c}.
\label{eq:d3-L}
\end{equation}

Finally, return to \eqref{eq:d3-collinear} and let \(y\downarrow0\), so
\(x=c-y\), whence
\[
S=\frac{y}{\sinh c}+o(y),
\qquad
g=1+\frac{1-\tau}{2}y+o(y),
\qquad
d=(1-\tau)+o(1),
\]
and \eqref{eq:d3-density} gives
\[
u(y,c,c-y)=L(c)+o(1).
\]
Comparing the coefficient of \(y\) gives
\begin{equation}
\sigma'(c)
=
\frac{1-\tau}{2}\sigma(c)
-
\frac{\tau}{\sinh c}L(c).
\label{eq:d3-sigma-second}
\end{equation}
Substituting \eqref{eq:d3-L} into \eqref{eq:d3-sigma-second} and comparing
with \eqref{eq:d3-sigma-first}, using \(\sigma'(0)=0\), gives
\[
(4e^{2c}+3e^c+1)\sigma(c)
=
4e^c(e^c+1)\sigma(0),
\]
so
\[
\sigma(c)
=
\sigma(0)
\frac{4e^c(e^c+1)}{4e^{2c}+3e^c+1}.
\]
Differentiating at \(c=0\) gives
$\sigma'(0)=\frac18\sigma(0)$.
Since \(\sigma'(0)=0\), we obtain \(\sigma(0)=0\); substituting this
back into the preceding explicit formula for \(\sigma(c)\) gives
$\sigma\equiv0$.
Formula \eqref{eq:d3-density} now gives
$
u\equiv0,$ as was claimed.

\medskip
\noindent
\textit{Step 3: The projected generator has a full additive completion.}
Set
\[
E_K^{PQR}=K(a,b,c)\,\beta_P\beta_Q\beta_R.
\]
For a cut \(T\in PQ\), define its full common-first subdivision defect by
\[
\begin{aligned}
\mathcal D_K(P,T,Q;R)
={}&E_K^{PQR}
-
\operatorname{Tr}_{PTR\to PQR}E_K^{PTR}\\
&-
\operatorname{Tr}_{TQR\to PQR}E_K^{TQR},
\end{aligned}
\]
where \(\operatorname{Tr}\) denotes the full common-first substitution from a
subtriangle to the parent triangle.  By Step 1, the cubic
\(\beta_P\beta_Q\beta_R\)-part cancels exactly, so
\(\mathcal D_K\) is of type \(\lambda\beta\beta\).  It is \(C^1\)-regular
in the cut parameter, and
\[
\mathcal D_K(P,P,Q;R)=
\mathcal D_K(P,Q,Q;R)=0
\]
as one-sided endpoint limits.

If \(P,T,U,Q\) occur in this order on \(PQ\), transport all terms to the
common-first coordinates of \(PQR\).  Associativity of the explicit
substitution formulas gives the two-cut identity
\[
\mathcal D_K(P,U,Q;R)+\mathcal D_K(P,T,U;R)
=
\mathcal D_K(P,T,Q;R)+\mathcal D_K(T,U,Q;R).
\]
This is exactly the same elementary identity used in the degree-two
completion.

Parametrize \(PQ\) by \(X_s\), \(0\le s\le c\).  For \(0\le r<t\le c\)
define
\[
\Gamma_K(r,t)
=
\left.
\frac{\partial}{\partial h}
\right|_{h=0^+}
\mathcal D_K(X_r,X_{r+h},X_t;R),
\]
and then
\[
\mathcal C_K(X_s,X_t,R)
=
\int_s^t
\operatorname{Tr}_{X_rX_tR\to X_sX_tR}
\Gamma_K(r,t)\,dr.
\]
The derivative exists by the stated \(C^1\) regularity and the endpoint
vanishing of \(\mathcal D_K\).  Differentiating the two-cut identity at the
collapsed daughter and applying the fundamental theorem of calculus, exactly
as in the degree-two proof, gives for \(s<u<t\)
\begin{equation}
\begin{aligned}
&\mathcal C_K(X_s,X_t,R)
-\mathcal C_K(X_s,X_u,R)
-\mathcal C_K(X_u,X_t,R)\\
&\qquad=
\mathcal D_K(X_s,X_u,X_t;R).
\end{aligned}
\label{eq:d3-correction-defect}
\end{equation}
It follows that
\begin{equation}
{
\cK_3^{PQR}=E_K^{PQR}-\mathcal C_K(P,Q,R)
}
\label{eq:K3-main}
\end{equation}
is additive under subdivision of \(PQ\).  The construction is equivariant
under the cyclic relabelling
\((P,Q,R)\mapsto(Q,R,P)\), so that \(\cK_3\) is cyclic, and its additivity for
the other two sides follows by cyclicity.  Since \(\mathcal C_K\) has type
\(\lambda\beta\beta\),we find
$
(\cK_3)_{\beta\beta\beta}=E_K$.
The integral defining \(\mathcal C_K\) vanishes on a collapsed marked side.
Writing its lower term cyclically defines the coefficient \(U_0\) appearing
in the proposition.

If two additive completions of \(E_K\) existed, their difference would have
zero cubic projection, and Step 2 would force that difference to vanish.
The additive completion is therefore unique. \end{proof}

\subsection{Degree four}
Continuity at the degenerate faces is enough to prove non-existence: additivity reduces the
coefficient to a one-dimensional boundary density, and cyclicity forces that
density to vanish.

\begin{proposition}[No degree-four generator]
\label{prop:degree-four-zero}
The only $C^0$-regular cyclic additive degree-four scalar invariant of hyperbolic spin
supertriangles vanishes.  
\end{proposition}

\begin{proof}
Since there are exactly four odd common-first coordinates, every degree-four
scalar expression has the form
\[
I_4^{PQR}
=
\lambda\beta_P\beta_Q\beta_R\,G(a,b,c),
\]
where cyclicity  implies
$
G(a,b,c)=G(b,c,a)=G(c,a,b)$.

In the standard notation of Section \ref{sec:subdivision-calculus}
modulo terms proportional to $\lambda$, the cut-point second fermion is
$
\beta_T=S\beta_P+U\beta_Q$.
The omitted lower terms in the transported $\beta$'s make no contribution
here, because the degree-four monomial already contains the factor $\lambda$
and hence every such correction would contain $\lambda^2=0$.

For the left daughter $PTR$, we therefore obtain
\[
g\lambda\,\beta_P
   (S\beta_P+U\beta_Q)\beta_R
=
gU\,\lambda\beta_P\beta_Q\beta_R,
\]
while for the right daughter $QRT$, we find
\[
d\lambda\,\beta_Q\beta_R
   (S\beta_P+U\beta_Q)
=
dS\,\lambda\beta_P\beta_Q\beta_R.
\]
Additivity therefore gives
$
G(a,b,c)
=
gU\,G(m,b,x)+dS\,G(m,y,a)$,
and using cyclicity in the second term, this becomes
\[
G(a,b,c)
=
gU\,G(m,b,x)+dS\,G(a,m,y).
\]
Define the auxiliary
\[
N(t)=\frac{e^{t/2}}{\sinh t\,\cosh(t/2)},
\qquad
\widetilde G(a,b,c)=\frac{G(a,b,c)}{N(c)}.
\]
Since
\[
gU=\frac{N(c)}{N(x)},
\qquad
dS=e^{-x}\frac{N(c)}{N(y)},
\]
the preceding additivity equation is equivalent to
\begin{equation}\label{eq:d4-weighted}
\widetilde G(a,b,c)
=
\widetilde G(m,b,x)
+
e^{-x}\widetilde G(a,m,y).
\end{equation}

Let
$
\sigma(t)=G(t,t,0)
$
be the regular collapsed-side trace.  Fix a non-degenerate triangle,
parametrize
\[
X_s\in PQ,
\qquad d(P,X_s)=s,
\qquad 0\le s\le c,
\]
by arc length, and define
$
r(s)=d(X_s,R)$.
Since
$
N(h)=\frac1h+O(1)
~~{\rm as}~ (h\downarrow0),
$
continuity at the degenerate face gives
\[
\widetilde G\bigl(d(X_{s+h},R),d(X_s,R),h\bigr)
=
h\,\sigma(r(s))+o(h).
\]
Repeated application of \eqref{eq:d4-weighted} to a partition of $PQ$,
followed by passage to the corresponding Riemann sum, yields
\begin{equation}\label{eq:d4-density}
\widetilde G(a,b,c)
=
\int_0^c e^{-s}\sigma(r(s))\,ds,
\end{equation}
or equivalently
\[
G(a,b,c)
=
N(c)\int_0^c e^{-s}\sigma(r(s))\,ds.
\]

It remains only to determine the one-variable trace $\sigma$.  First, take the
degenerate triangle
$
(a,b,c)=(t,0,t)$,
for which $R=P$ and hence $r(s)=s$.  By cyclicity,
$G(t,0,t)=G(t,t,0)=\sigma(t)$, so \eqref{eq:d4-density} gives
\[
\sigma(t)
=
N(t)\int_0^t e^{-s}\sigma(s)\,ds.
\]
Next, take $(a,b,c)=(0,t,t)$,
for which $R=Q$ and $r(s)=t-s$.  Again by cyclicity,
$G(0,t,t)=\sigma(t)$, and therefore
\[
\sigma(t)
=
N(t)e^{-t}\int_0^t e^s\sigma(s)\,ds.
\]
Setting
\[
A(t)=\int_0^t e^{-s}\sigma(s)\,ds,
\qquad
B(t)=\int_0^t e^s\sigma(s)\,ds,
\]
the last two identities imply
$
A(t)=e^{-t}B(t)$.
Since $\sigma$ is continuous, the fundamental theorem of calculus applies.
Differentiating gives
$
e^{-t}\sigma(t)
=
-e^{-t}B(t)+\sigma(t)$,
and hence
\[
A(t)=e^{-t}B(t)
=(1-e^{-t})\sigma(t).
\]
Combining this with $\sigma(t)=N(t)A(t)$ gives
\[
\sigma(t)
=
N(t)(1-e^{-t})\sigma(t)
=
\operatorname{sech}^2(t/2)\,\sigma(t).
\]
For every $t>0$,
\(
\operatorname{sech}^2(t/2)<1
\), so
$
\sigma(t)=0$.
Continuity gives $\sigma(0)=0$ as well.  Finally,
\eqref{eq:d4-density} gives
$
\widetilde G\equiv0,
~
G\equiv0$,
and therefore
$
I_4\equiv0$, as required.
\end{proof}

\section*{Appendix: Examples and special families}
\label{sec:worked-examples-generators}

We illustrate the four generators on several families for which their
geometric meaning is especially transparent.  The first example displays the
spin involution itself; the second is the particularly simple locus
$\lambda=0$; the third exploits an isosceles body; and the last relates the
degree-two invariant $J$ to geodesics with ideal endpoints in the special
light cone.  Throughout, a spin lift is fixed unless explicitly stated
otherwise.

We use the common-first form
\[
P=((1-\lambda\beta_P)y_P\mid\lambda,\beta_P),\qquad
Q=((1-\lambda\beta_Q)y_Q\mid\lambda,\beta_Q),\qquad
R=((1-\lambda\beta_R)y_R\mid\lambda,\beta_R),
\]
with body side lengths
\[
a=d(Q_\#,R_\#),\qquad b=d(R_\#,P_\#),\qquad c=d(P_\#,Q_\#).
\]
The degree-two side corrections are
\[
\Delta a=-\lambda(\beta_Q+\beta_R)\tanh(a/2),
\]
\[
\Delta b=-\lambda(\beta_R+\beta_P)\tanh(b/2),
\qquad
\Delta c=-\lambda(\beta_P+\beta_Q)\tanh(c/2),
\]
and therefore
\[
J=d\Area(\Delta a,\Delta b,\Delta c).
\]

\begin{example}[Spin reversal]
\label{ex:spin-reversal}
{\rm Let \(\Delta^+\) be a spin supertriangle and \(\Delta^-\) the same
underlying supertriangle with the opposite spin lift.  Then
\[
\Area(\Delta^-)=\Area(\Delta^+),\qquad
J(\Delta^-)=J(\Delta^+),
\]
\[
\cB_2(\Delta^-)=\cB_2(\Delta^+),\qquad
\cK_3(\Delta^-)=-\cK_3(\Delta^+).
\]
Thus, the first three generators descend to ordinary supertriangles, while the
last one lives naturally on the spin double cover.  This is the simplest
concrete manifestation of the distinction between the two classification
problems.}
\end{example}

\begin{example}[The family \(\lambda=0\)]
\label{ex:lambda-zero}
{\rm If
$
\lambda=0$, then
\[
\Delta a=\Delta b=\Delta c=0,
\qquad
J=0.
\]
Moreover every lower term in the full common-first transport of the second
fermions contains a factor of \(\lambda\), so on every daughter triangle
arising in a subdivision, full transport agrees with projected transport.
The subdivision defects of the projected degree-two and degree-three
expressions therefore vanish identically, and the additive lifts are exactly
the projected expressions
\[
{
\cB_2
=
\beta_P\beta_Q\frac{c}{\sinh c}
+
\beta_Q\beta_R\frac{a}{\sinh a}
+
\beta_R\beta_P\frac{b}{\sinh b}
}
\]
and
\[
{
\cK_3
=
\beta_P\beta_Q\beta_R\,K(a,b,c),
}
\]
where
\[
K(a,b,c)
=
\frac1{\sinh c}
\int_{P_\#Q_\#}
\frac{2\,ds}{1+\cosh d(X,R_\#)}=
\frac{2\Area(P_\#,Q_\#,R_\#)}{\sinh h_c\,\sinh c}\,,
\]
where \(h_c\) denotes the signed body altitude from \(R_\#\) to the geodesic
\(P_\#Q_\#\). This family simultaneously displays the two new generators in their simplest
form.  Under spin reversal the quadratic expression \(\cB_2\) is unchanged,
whereas the cubic expression \(\cK_3\) changes sign.}
\end{example}

\begin{example}[An isosceles body]
\label{ex:isosceles-body}
{\rm Put the body side \(P_\#Q_\#\) in the standard position
\[
y_P=(1,1,0),\qquad
y_Q=(e^c,e^{-c},0),
\]
and place the third body vertex 
\[
y_R=
\bigl(\cosh h\,e^{c/2},\ \cosh h\,e^{-c/2},\ \sinh h\bigr),
\]
over the midpoint of this side, where $h>0$.
It follows that
\[
a=b=\ell,
\qquad
\cosh\ell=\cosh h\cosh(c/2).
\]
If \(A=B\) are the two base angles and \(C\) is the apex angle, then
\[
\cos A=\coth\ell\,\tanh(c/2),
\qquad
\cos C=
\frac{\cosh^2\ell-\cosh c}{\sinh^2\ell},
\]
so
\[
\Omega_{\ell,c}=\Area(P_\#,Q_\#,R_\#)
=\pi-2A-C.
\]

At the symmetric point \((a,b,c)=(\ell,\ell,c)\), write
\[
\Omega_\ell
=
\left.\frac{\partial\Omega}{\partial a}\right|_{a=b=\ell}
=
\left.\frac{\partial\Omega}{\partial b}\right|_{a=b=\ell},
\qquad
\Omega_c
=
\left.\frac{\partial\Omega}{\partial c}\right|_{a=b=\ell},
\]
so the defect-transgression generator is
\[
{
\begin{aligned}
J={}&-\lambda\,\Omega_\ell\tanh(\ell/2)
\bigl(\beta_P+\beta_Q+2\beta_R\bigr)\\
&-\lambda\,\Omega_c\tanh(c/2)
\bigl(\beta_P+\beta_Q\bigr).
\end{aligned}}
\]
The projected degree-two part is
\[
E_B
=
\beta_P\beta_Q\frac{c}{\sinh c}
+
\frac{\ell}{\sinh\ell}
\bigl(\beta_Q\beta_R+\beta_R\beta_P\bigr),
\]
while the projected cubic part is
\[
E_K
=
\beta_P\beta_Q\beta_R\,
\frac{2\Omega_{\ell,c}}{\sinh h\,\sinh c}.
\]
The reflection symmetry of the body is visible directly in all three
fermionic sectors.  For arbitrary \(\lambda\), the full invariants
\(\cB_2\) and \(\cK_3\) are obtained from these projected terms by the direct
subdivision-defect corrections of the main proof.  If in addition
\(\lambda=0\), those corrections vanish and the displayed \(E_B,E_K\) are
already the complete values of \(\cB_2,\cK_3\).}
\end{example}

\begin{example}[A supporting geodesic with special ideal endpoints]
\label{ex:special-light-cone-side}
{\rm Recall that the special light cone is
\[
\LLp_0=\osp\cdot e_0\subset\LLp,
\qquad
e_0=(1,0,0\mid0,0).
\]
Equivalently, it is the locus on which the fermion label $\theta$ of the light-cone
orbit vanishes.  More explicitly for
\(A=(x_1,x_2,y\mid\varphi,\psi)\in\LLp\), Proposition~2.5 of \cite{PZ} gives,
$
\theta(A)
=
\sqrt{x_1}\,\psi-\frac{y}{\sqrt{x_1}}\,\varphi
$
when \(x_1\) is invertible, and equivalently
$
\theta(A)
=
\sqrt{x_2}\,\varphi-\frac{y}{\sqrt{x_2}}\,\psi
$
when \(x_2\) is invertible, in each case up to the common fermionic sign.
Membership in \(\LLp_0\) is the intrinsic condition
\(\theta(A)=0\), and does not mean that the displayed odd coordinates vanish
in an arbitrary frame.

Suppose the two ideal endpoints of the supporting geodesic of the oriented
side \(PQ\) lie in \(\LLp_0\).  Since an independent ordered pair in
\(\LLp_0\) is classified by its pairing, cf.\ \cite{PZ}, after an
\(\osp\)-transformation preserving the chosen spin lift the two endpoint
vectors may be taken to be
\[
U=(2,0,0\mid0,0),
\qquad
V=(0,2,0\mid0,0),
\qquad
\pair{U}{V}=2.
\]
The supporting geodesic is therefore
\[
X_s=\frac12(e^sU+e^{-s}V)
=(e^s,e^{-s},0\mid0,0).
\]
Write the opposite vertex as
\[
R=(1-\alpha\beta)R_0\mid\alpha,\beta,
\]
with
\[
R_0=
\bigl(\cosh d\,e^\rho,\ \cosh d\,e^{-\rho},\ \sinh d\bigr),
\qquad
C=\cosh d,
\qquad
S=\sinh d.
\]
Set
\[
A_0=Ce^{-\rho},
\qquad
B_0=Ce^{\rho},
\qquad
c_s=\pair{X_s}{R_0}
=\frac12(e^sA_0+e^{-s}B_0)
=C\cosh(s-\rho).
\]
In this special-light-cone normal form, the endpoint fermions vanish.  The
local density for the degree-two transgression \(J\) therefore reduces to the
single opposite-vertex sector
\[
j(s)
=
-\alpha\beta\,
\frac{c_s+A_0B_0}{S(c_s+1)^2},
\]
which has the elementary potential
\[
\mathcal F(z)
=
-\alpha\beta\,
\frac{A_0z^2-B_0}
{S(A_0z^2+2z+B_0)},
\qquad z=e^s,
\]
with
\[
z\frac{d}{dz}\mathcal F(z)=j(s).
\]
Consequently, if \(P=X_{s_0}\) and \(Q=X_{s_1}\), then
\[
J(P,Q,R)
=
\mathcal F(e^{s_1})-\mathcal F(e^{s_0}).
\]
This example is useful for two reasons.  First, it shows that having
special-light-cone ideal endpoints does not force the fermionic defect
transgression to vanish: the odd data at the opposite vertex can still
contribute through the even product \(\alpha\beta\).  Second, spin reversal
sends \((\alpha,\beta)\mapsto(-\alpha,-\beta)\), so the product
\(\alpha\beta\), and therefore \(J\), is unchanged as required.

If all three supporting geodesics of the triangle have their ideal endpoints
in \(\LLp_0\), the same calculation applies separately to each oriented side.
One uses a separate special-light-cone normal form for each side; the resulting
three endpoint-difference formulas agree because \(J\) is an intrinsic cyclic
invariant.  No conclusion such as \(\lambda=0\) should be drawn from the
special-ended condition alone: that condition concerns the fermion labels of
the ideal endpoints of the supporting geodesics, whereas \(\lambda\) is the
common-first fermion of the finite spin supertriangle.}
\end{example}

\bibliographystyle{plain}
\bibliography{ref.bib}

@article{corrigendum,
	author = {Penner, Robert C.},
	journal = {Journal of Geometry and Physics},
	title = {Corrigendum to {\it 3d super hyperbolic geometry} with an Appendix correcting {\it Angle defect for super triangles}},
	year = {2026}}

@article{Klain2006HyperbolicValuations,
	author = {Klain, Daniel A.},
	doi = {10.1007/s00454-006-1251-6},
	journal = {Discrete \& Computational Geometry},
	number = {3},
	pages = {457--477},
	title = {Isometry-Invariant Valuations on Hyperbolic Space},
	volume = {36},
	year = {2006}}

@incollection{lambert,
	address = {Leipzig},
	author = {Lambert, Johann Heinrich},
	booktitle = {Die Theorie der Parallellinien von Euklid bis auf Gauss: Eine Urkundensammlung zur Vorgeschichte der nichteuklidischen Geometrie},
	editor = {Engel, Friedrich and St{\"a}ckel, Paul},
	note = {Written in 1766},
	pages = {152--208},
	publisher = {B. G. Teubner},
	title = {Theorie der Parallellinien},
	year = {1895}}

@article{Dehn1900,
	author = {Dehn, Max},
	journal = {Nachrichten von der K{\"o}niglichen Gesellschaft der Wissenschaften zu G{\"o}ttingen, Mathematisch-Physikalische Klasse},
	pages = {345--354},
	title = {Ueber raumgleiche Polyeder},
	year = {1900}}

@article{Zakharevich2016,
	author = {Zakharevich, Inna},
	doi = {10.1090/bull/1527},
	journal = {Bulletin of the American Mathematical Society},
	number = {2},
	pages = {269--294},
	title = {Perspectives on Scissors Congruence},
	volume = {53},
	year = {2016}}

@article{Hilbert1900,
	author = {Hilbert, David},
	doi = {10.1090/S0002-9904-1902-00923-3},
	journal = {Bulletin of the American Mathematical Society},
	number = {10},
	pages = {437--479},
	title = {Mathematical Problems},
	translator = {Newson, Mary Winston},
	volume = {8},
	year = {1902}}

@article{Sydler1965,
	author = {Sydler, Jean-Pierre},
	doi = {10.1007/BF02564364},
	journal = {Commentarii Mathematici Helvetici},
	pages = {43--80},
	title = {Conditions n{\'e}cessaires et suffisantes pour l'{\'e}quivalence des poly{\`e}dres de l'espace euclidien {\`a} trois dimensions},
	volume = {40},
	year = {1965}}

@article{3d,
	author = {Penner, Robert C.},
	doi = {10.1016/j.geomphys.2024.105171},
	journal = {Journal of Geometry and Physics},
	pages = {105171},
	title = {3d Super Hyperbolic Geometry},
	volume = {199},
	year = {2024}}

@article{IPZ,
	archiveprefix = {arXiv},
	author = {Ip, Ivan C. H. and Penner, Robert C. and Zeitlin, Anton M.},
	doi = {10.1016/j.aim.2018.07.019},
	eprint = {1605.08094},
	journal = {Advances in Mathematics},
	pages = {409--454},
	primaryclass = {math.GT},
	title = {{$\mathcal{N}=2$} super-{T}eichm{\"u}ller space},
	volume = {336},
	year = {2018}}

@article{PZ,
	archiveprefix = {arXiv},
	author = {Penner, R. C. and Zeitlin, Anton M.},
	doi = {10.4310/jdg/1552442609},
	eprint = {1509.06302},
	journal = {Journal of Differential Geometry},
	number = {3},
	pages = {527--566},
	primaryclass = {math.GT},
	title = {Decorated super-{T}eichm{\"u}ller space},
	volume = {111},
	year = {2019}}

@incollection{defect,
	address = {Cham},
	archiveprefix = {arXiv},
	author = {Penner, Robert C.},
	booktitle = {Essays on Geometry},
	doi = {10.1007/978-3-031-76257-4_6},
	editor = {Uluda\u{g}, A. Muhammed and Zeytin, Ayberk},
	eprint = {2208.06826},
	pages = {81--109},
	primaryclass = {math.GT},
	publisher = {Springer},
	title = {Angle Defect for Super Triangles},
	year = {2025}}

\end{document}